\documentclass[12pt]{amsart}
\usepackage[margin=2.5cm]{geometry}
\usepackage{tikz-cd,stmaryrd,amssymb}
\usepackage{hyperref}

\DeclareFontFamily{OT1}{rsfs}{}
 \DeclareFontShape{OT1}{rsfs}{n}{it}{<->rsfs10}{}
 \DeclareMathAlphabet{\curly}{OT1}{rsfs}{n}{it}

\theoremstyle{plain}  
\newtheorem{theorem}{Theorem}[section]

\newtheorem*{theorem*}{Theorem}
\newtheorem{corollary}[theorem]{Corollary}
\newtheorem{lemma}[theorem]{Lemma}
\newtheorem{proposition}[theorem]{Proposition}

\theoremstyle{definition}
\newtheorem{definition}[theorem]{Definition}

\theoremstyle{remark}

\newtheorem{remark}[theorem]{Remark}
\newtheorem*{remark*}{Remark}

\newtheorem*{claim*}{Claim}

\renewcommand{\phi}{\varphi}

\newcommand{\andd}{\quad\text{and}\quad}
\newcommand{\suhthat}{\;:\;}

\newcommand{\geqp}{%
  \mathrel{\raisebox{-0.5ex}{$\scriptscriptstyle($}}%
  \geq
  \mathrel{\raisebox{-0.5ex}{$\scriptscriptstyle)$}}%
}

\newcommand{\N}{\mathbb{N}}
\newcommand{\Q}{\mathbb{Q}}
\newcommand{\R}{\mathbb{R}}

\newcommand{\Z}{\mathbb{Z}}
\newcommand{\C}{\mathbb{C}}
\newcommand{\E}{\mathbb{E}}
\newcommand{\Eb}{\E_{\bullet}}

\DeclareMathOperator{\Aut}{Aut}
\newcommand{\SL}{\mathrm{SL}}
\newcommand{\GL}{\mathrm{GL}}
\newcommand{\tg}{\theta_{\gamma}}
\DeclareMathOperator{\Int}{Int}

\newcommand{\lie}{\mathfrak}
\DeclareMathOperator{\Out}{Out}

\newcommand{\zk}{\lie z_{\lie k}}
\newcommand{\gl}{\lie{gl}(n,\C)}

\DeclareMathOperator{\stab}{Stab}
\newcommand{\Gm}{\mathbb{G}_m}
\newcommand{\tP}{\widetilde{P}}
\newcommand{\tL}{\widetilde{L}}
\newcommand{\cha}{\zeta}

\newcommand{\ttau}{\widetilde{\tau}}

\newcommand{\oo}{\mathcal{O}}
\newcommand{\iso}{\mathcal{I}so}
\newcommand{\hhom}{\mathcal{H}om}
\newcommand{\sym}{\mathcal{S}ym}
\DeclareMathOperator{\Sym}{Sym}
\DeclareMathOperator{\id}{id}

\newcommand{\eb}{E_{\bullet}}
\newcommand{\ebt}{E_{t\bullet}}
\newcommand{\alb}{\alpha_{\bullet}}
\DeclareMathOperator{\rk}{rk}

\newcommand{\Cabc}{\C^n_{a,b,c}}
\newcommand{\rhoabc}{\rho_{a,b,c}}

\newcommand{\Vg}[1]{V_{\Gamma,#1}}

\newcommand{\bundle}{Q}

\newcommand{\qq}{\mathfrak Q}
\newcommand{\oqq}{\overline{\qq}}
\newcommand{\f}{\mathfrak F}
\newcommand{\ii}{\mathfrak{I}}
\newcommand{\jj}{\mathfrak J}
\newcommand{\cc}{\mathfrak C}
\newcommand{\spec}{\mathcal{S}pec}
\DeclareMathOperator{\Spec}{Spec}

\newcommand{\fg}{f_{\gamma}}
\newcommand{\action}{\bullet}

\newcommand{\PP}{\mathbb{P}}

\DeclareMathOperator{\Hom}{Hom}

\newcommand{\pair}[2]{\langle#1,#2\rangle}

\newcommand{\trho}{\widetilde{\rho}}
\newcommand{\M}[1]{M_{#1}(\C)}
\newcommand{\MM}[1]{M_{#1}(\C)\times M_{#1}(\C)}
\newcommand{\GLL}{\GL(n,\C)\times \GL(n,\C)}
\newcommand{\GLX}{\GL(n,\C(X))}

\DeclareMathOperator{\Jac}{Jac}

\newcommand{\Ll}{\mathcal{L}}
\newcommand{\Pp}{\mathcal{P}}

\newcommand{\Gg}{\mathcal{G}{(m)}}
\newcommand{\GG}{\mathbb{G}{(m)}}
\newcommand{\GGss}{\GG^{ss}}
\DeclareMathOperator{\Gies}{Gies}
\newcommand{\Giesss}{\Gies^{ss}}

\newcommand{\ccss}{\cc_0^{ss}}

\newcommand{\mumaxg}{\mu_{\text{max}}^{\Gamma}}
\newcommand{\mumin}{\mu_{\text{min}}^{\Gamma}}

\newcommand{\tsigma}{\widetilde{\sigma}}
\newcommand{\tE}{\widetilde{E}}
\newcommand{\cxz}{(C\times X)\setminus (C\times Z)}
\newcommand{\tf}{\widetilde{f}}

\newcommand{\Mm}{\mathcal{M}}
\newcommand{\Ff}[1]{\mathcal F_{#1}}
\newcommand{\tFf}[1]{\widetilde{\mathcal F}_{#1}}
\DeclareMathOperator{\pt}{pt}
\newcommand{\tphi}{\widetilde{\phi}}
\newcommand{\teta}{\widetilde{\eta}}

\newcommand{\HG}{H^1_{\theta}(\Gamma,G)}

\newcommand{\tJac}{\widetilde{\Jac}}

\newcommand{\otheta}{\overline{\theta}}
\newcommand{\oc}{\overline{c}}
\DeclareMathOperator{\Gal}{Gal}

\newcommand{\olambda}{\overline{\lambda}}

\title{Moduli spaces of twisted equivariant \texorpdfstring{$G$}{G}-bundles over a curve}
\author{G. Barajas}

\date{\today}



\thanks{The author has received the support of the ERC Synergy Grant 810573/2019.}

\begin{document}

\begin{abstract}
Let $X$ be a compact Riemann surface, $\Gamma$ a finite group of automorphisms of $X$ and $G$ a connected reductive complex Lie group with center $Z$. If we equip this data with a homomorphism $\theta:\Gamma\to\Aut(G)$ and a 2-cocycle $c\in Z^2_{\theta}(\Gamma,Z)$, there is a notion of $(\theta,c)$-twisted $\Gamma$-equivariant $G$-bundle over $X$. The aim of this paper is to construct a coarse moduli space of isomorphism classes of polystable $(\theta,c)$-twisted equivariant $G$-bundles over $X$, according to the definition of polystability given by Garc\'ia-Prada--Gothen--Mundet i Riera. This generalizes the well-known construction of the moduli space of $G$-bundles given by Ramanathan. It also gives, in particular, a GIT construction of the moduli space of $\Gamma$-equivariant $G$-bundles, and the moduli space of $\hat G$-bundles for $\hat G$ non-connected by our joint work with Garc\'ia-Prada, Gothen and Mundet i Riera --- complementing the construction of a projective good moduli space for the moduli stack of $\hat G$-bundles given by Olsson--Reppen--Tajakka.


\end{abstract}

\maketitle

\section{Introduction}

Let $X$ be a compact Riemann surface, $\Gamma$ a finite group of automorphisms of $X$ and $G$ a connected reductive complex Lie group with center $Z$. Fix a homomorphism $\theta:\Gamma\to\Aut(G)$. Recall that there is a group of 2-cocycles $Z^2_{\theta}(\Gamma,Z)$, appearing for example in the context of Galois cohomology \cite{serre-galois}, consisting of maps $\Gamma\times\Gamma\to Z$ satisfying a precise condition (\ref{eq-2-cocycle}). Pick a 2-cocycle $c\in Z^2_{\theta}(\Gamma,Z)$. 

Given a $G$-bundle $\bundle\to X$, a $(\theta,c)$-twisted $\Gamma$-equivariant structure on $\bundle$ is a lift of the $\Gamma$-action on $X$ to a map from $\Gamma$ to the group of holomorphic automorphisms of the total space of $\bundle$. This action must further twist the bundle $G$-action via $\theta$, and the composition of operating by $\gamma_1$ and $\gamma_2\in\Gamma$ must be the same as multiplying by $c(\gamma_1,\gamma_2)$ and then operating by $\gamma_1\gamma_2$. The condition of $c$ being a 2-cocycle is equivalent to the associativity of this twisted equivariant structure. 

A \textbf{$(\theta,c)$-twisted $\Gamma$-equivariant $G$-bundle} over $X$ is a pair $(\bundle,\action)$ consisting of a $G$-bundle $\bundle$ and a $(\theta,c)$-twisted $\Gamma$-equivariant structure $\action$ on $\bundle$. An isomorphism of twisted equivariants $G$-bundles is just a $\Gamma$-equivariant isomorphism of $G$-bundles. \textbf{For some motivation on why we study these objects, see the last five paragraphs of the introduction.}

A definition of $\zeta$-(poly, semi)-stability for twisted equivariant $G$-bundles --- which are, in particular, twisted equivariant $G$-Higgs bundles --- appears naturally in the work of Garc\'ia-Prada--Gothen--Mundet i Riera \cite{oscar-ignasi-gothen}. This definition fits into a Hitchin--Kobayashi correspondence for Higgs bundles with non-connected structure group, and is compatible with the description of fixed point subvarieties in moduli spaces principal bundles according to our joint work with Garc\'ia-Prada--Basu \cite{oscar-barajas-higgs}. It depends on a parameter $\zeta$, which may be thought of as the topological type of the underlying $G$-bundle.
This paper is dedicated to the construction of a projective variety parametrizing isomorphism classes of polystable $(\theta,c)$-twisted $\Gamma$-equivariant $G$-bundles, using Mumford's GIT \cite{mumford-GIT}. In particular, it shows that the (poly, semi)stability notions in the literature are also compatible with the corresponding GIT notions. The main result is the following.

\vspace{5 pt}
{\bf Main Theorem} (Theorem \ref{th-moduli-space-G-reductive}). {\it There exists a complex projective variety $\Mm(X,G,\Gamma,\theta,c)$ which is a coarse moduli space parametrizing isomorphism classes of polystable $(\theta,c)$-twisted $\Gamma$-equivariant $G$-bundles over $X$. It contains an open subvariety classifying isomorphism classes of $\cha$-stable $(\theta,c)$-twisted $\Gamma$-equivariant $G$-bundles over $X$.}
\vspace{5 pt}

As a corollary, using an equivalence of categories between principal bundles with non-connected structure group $\hat G$ and twisted equivariant bundles with structure group equal to the connected component of the identity $G$ --- see \cite{BGGM} ---, we obtain the following.

\vspace{5 pt}
{\bf Corollary} (Corollary \ref{cor-moduli-space-G-non-connected}). {\it  
Let $\hat G$ be an arbitrary --- possibly non-connected --- reductive complex Lie group, and let $\tilde X\to X$ be a $\Gamma$-bundle. Then there is a complex projective variety $\Mm_{\tilde X}(X,\hat G)$ which is a coarse moduli space parametrizing isomorphism classes of polystable $\hat G$-bundles over $X$ whose quotient by $G$ is isomorphic to $\tilde X$. It contains an open subvariety classifying isomorphism classes of stable $\hat G$-bundles over $X$.}
\vspace{5 pt}

This corollary is a GIT complement of the construction of a projective good moduli space for the moduli stack of $\hat G$-bundles given by Olsson--Reppen--Tajakka \cite{reppen}.

In order to prove the main theorem, we first construct a coarse moduli space for $G$ \textbf{semisimple}. In order to do this we need to formulate the problem in terms of GIT. First we fix an embedding $\iota:G\hookrightarrow\GL(n,\C)$ such that $\Gamma$ acts on $G$ via inner automorphisms of $\GL(n,\C)$. In Sections \ref{section-twisted-equivariant-and-pseudo-equivariant} and \ref{section-pseudo-equivariant-bundles-as-pseudo-equivariant-pairs} we show that the category of (semi)stable $(\theta,c)$-twisted $\Gamma$-equivariant $G$-bundles over $X$ may be embedded into the category of \textbf{$(\delta,\chi)$-(semi)stable pseudo-equivariant $\rho$-pairs}, for $\chi\gg\delta\gg0$ and a suitable representation $\rho:\GL(n,\C)\to\GL(W)$. A pseudo-equivariant $\rho$-pair is a triple $(E,\sigma,f)$, where $E$ is a vector bundle of rank $n$ and trivial determinant, $\sigma:E(W)\to \oo_X$ is a non-zero homomorphism and $f$ is a $\Gamma$-equivariant structure on $E$ satisfying certain compatibility condition (\ref{eq-pseudo-equivariant-pairs-f-sigma}) with $\sigma$. Here $E(W)$ is the associated vector bundle with fibre $W$. The notion of $(\delta,\chi)$-(semi)stability, which depends on two positive parameters $\delta$ and $\chi$, is defined using 1-parameter subgroups and weighted filtrations of $E$ --- see Definition \ref{def-delta,chi-semistable-pairs}.

Using pseudo-equivariant pairs and Grothendieck's Quot Scheme parametrizing quotients of $V\otimes\oo_X$ of rank $n$ and trivial determinant for some fixed vector space $V$, we build a parameter space $\cc$ of pseudo-equivariant pairs containing all the (semi)stable $(\theta,c)$-twisted $\Gamma$-equivariant $G$-bundles. Moreover, $\cc$ is equipped with a $\GL(V)$-action such that two points are isomorphic if and only if they are in the same orbit --- see Proposition \ref{prop-iso-vs-orbit}. We then construct an injective morphism to a projective space $\GG$, called the \textbf{Gieseker space} --- which depends on a positive integer $m$ --- together with a compatible $\GL(V)$-action and a suitable linearization on some very ample line bundle --- see Section \ref{section-gieseker}. With these ingredients, Mumford's GIT provides notions of (semi)stability which we show to be compatible with the notions of $(\delta,\chi)$-(semi)stability for pseudo-equivariant pairs if $\chi\gg\delta\gg0$ and $m\gg0$. We conclude that there exists a quasi-projective subvariety $\ccss\subset\cc$ of $\cc$ parametrizing semistable $(\theta,c)$-twisted $\Gamma$-equivariant $G$-bundles --- see Corollary \ref{cor-parameter-space-semistable-twisted-equivariant}. The GIT quotient $\ccss\sslash\GL(V)$ is the required moduli space of polystable $(\theta,c)$-twisted $\Gamma$-equivariant $G$-bundles, for $G$ semisimple --- see Theorem \ref{th-moduli-space-G-semisimple}.

We then proceed to give a construction of a fine moduli space for $G$ \textbf{abelian} --- see Theorem \ref{th-moduli-space-G-abelian} ---, and then use the semisimple and abelian cases to prove the main Theorem following a technique introduced by Ramanathan \cite{ramanathan1} and generalized by G\'omez--Sols \cite{gomez-sols}.

Here is the outline of the paper. In Section \ref{section-twisted-equivariant-bundles} we define $(\theta,c)$-twisted $\Gamma$-equivariant $G$-bundles and the corresponding notions of $\zeta$-(poly, semi)stability. In Section \ref{section-twisted-equivariant-and-pseudo-equivariant} we prove an equivalence of categories between $(\theta,c)$-twisted $\Gamma$-equivariant $G$-bundles and pseudo-equivariant $G$-bundles for $G$ semisimple. These are triples $(E,h,f)$ consisting of a vector bundle $E$ of rank $n$, a reduction of structure group $h$ of its bundle of frames to $G$ and a $\Gamma$-equivariant structure on $E$ satisfying a compatibility condition (\ref{eq-pseudo-equivariant-f-h}) with $h$. In Section \ref{section-semistable-pseudo-equivariant-pairs} we define pseudo-equivariant $\rho$-pairs and the corresponding notion of $(\delta,\chi)$-(semi)stability. In Section \ref{section-pseudo-equivariant-bundles-as-pairs} we embed the category of pseudo-equivariant $G$-bundles into the category of pseudo-equivariant $\rho$-pairs, where $\rho=\rhoabc:\GL(n,\C)\to\GL\left(\left(\left(\C^n\right)^{\otimes a}\right)^{\oplus b}\otimes(\bigwedge^n\C^n)^{\otimes-c}\right)$ is induced by the standard representation, for some positive integers $a,b$ and $c$. In Section \ref{section-semistable-pseudo-equivariant-as-semistable-pairs} we show that this restricts to an embedding of the category of (semi)stable pseudo-equivariant $G$-bundles into the category of $(\delta,\chi)$-(semi)stable pseudo-equivariant $\rho$-pairs, for $\chi\gg\delta\gg0$ --- see Proposition \ref{prop-semistable-pseudo-vs-pairs}.

In Section \ref{section-parameter-space} we build the parameter space $\cc$ of pseudo-equivariant pairs, equipped with a $\GL(V)$-action. In Sections \ref{section-gieseker} and \ref{section-hilbert-mumford} we find an injection into a Gieseker space, and a suitable linearization of the $\GL(V)$-action on a very ample line bundle such that GIT (semi)stability is equivalent to $(\delta,\chi)$-(semi)stability for $\chi\gg\delta\gg0$. In Section \ref{section-projective} we show that the GIT quotient $\Mm(X,G,\Gamma,\theta,c):=\ccss\sslash\GL(V)$ is projective, where $\ccss\subset\cc$ is the subset of semistable $(\theta,c)$-twisted $\Gamma$-equivariant $G$-bundles. In Section \ref{section-polystable} we show that closed $\GL(V)$-orbits in $\ccss$ are the orbits of polystable objects, hence $\Mm(X,G,\Gamma,\theta,c)$ parametrizes isomorphism classes of polystable $(\theta,c)$-twisted $\Gamma$-equivariant $G$-bundles. 

In Section \ref{section-G-abelian} we build a fine moduli space of polystable $(\theta,c)$-twisted $\Gamma$-equivariant $G$-bundles for $G$ abelian. This, together with the result for $G$ semisimple, allows us to construct the moduli space $\Mm(X,G'\times Z,\Gamma,\theta,c)$, where $G'$ is semisimple and $Z$ is abelian. Finally, using Lemma \ref{lemma-R1-from-R2}, in Section \ref{section-moduli-G-reductive} we construct the moduli space $\Mm(X,G,\Gamma,\theta,c)$ of polystable $(\theta,c)$-twisted $\Gamma$-equivariant $G$-bundles for an arbitrary connected reductive complex Lie group $G$. The moduli space $\Mm(X,G,\Gamma,\theta,c)$ is given as a Galois cover of $\Mm(X,G'\times Z,\Gamma,\theta,c)$, where $G'=[G,G]$ and $Z$ is the centre of $G$.



Beyond its intrinsic geometric interest, we believe that a construction of the moduli space of  polystable \((\theta, c)\)-twisted $\Gamma$-equivariant \(G\)-bundles was overdue due to their important role in the literature. For example, the category of twisted equivariant $G$-bundles for $\theta$ and $c$ trivial is just the category of $\Gamma$-equivariant $G$-bundles. This has been extensively studied, for example in relation to \textbf{parabolic bundles} in the works of Balaji--Seshadri \cite{balaji-seshadri}, Biswas \cite{biswas} and Teleman--Woodward \cite{teleman}, and in the contex of the \textbf{geometric Langlands programme} in the works of Ben-Zvi--Francis--Nadler \cite{nadler1} and Chriss--Ginzburg \cite{ginzburg}. In a paper in progress with Gallego \cite{gallego}, we will prove an equivalence of categories between \((\theta, c)\)-twisted $\Gamma$-equivariant \(G\)-bundles over $X$ and parabolic bundles twisted by a gerbe over $X/\Gamma$, for general $\theta$ and $c$.

When $c$ is trivial, the moduli stack of $\theta$-twisted $\Gamma$-equivariant $G$-bundles appears in relation with \textbf{parahoric bundles} in the work of Damiolini \cite{damiolini1} and Damiolini--Hong \cite{damiolini2}. Pappas--Rapoport \cite{pappas-rapoport} consider a more general setup where $G$ is replaced by a --- non-necessarily constant --- group scheme $\mathcal G'$ over $X$, in relation to tamely ramified bundles. The category of $\theta$-twisted $\Gamma$-equivariant bundles also plays an important role in the study of the moduli space of \textbf{$G$-Higgs bundles} by Donagi--Gaitsgory \cite{donagi-gaitsgory}.

In full generality, $(\theta,c)$-twisted $\Gamma$-equivariant $G$-bundles appear naturally in the study of \textbf{fixed point varieties of finite group actions} on the moduli space $\Mm(X,G)$ of $G$-bundles according to our joint work with Garc\'ia-Prada--Basu \cite{oscar-barajas-higgs}. More precisely, it is straightforward to see that $(\theta,c)$-twisted $\Gamma$-equivariant $G$-bundles come up when considering the action of $\Gamma$ on $\Mm(X,G)$ such that $\gamma\in\Gamma$ sends a $G$-bundle $Q$ to $\gamma^*\theta_{\gamma}^{-1}(Q)$, given by pullback and extension of structure group. Perhaps it is less straightforward to see that twisted equivariant bundles --- on an étale cover of $X$ --- appear naturally as fixed points for the action --- on $\Mm(X,G)$ --- of a finite subgroup $\Lambda$ of $Z$-bundles over $X$ by ``tensorization'', via a \textbf{Prym--Narasimhan--Ramanan construction}. This is proved in our joint paper with García--Prada \cite{PNR}, which is a generalization of the construction given by Narasimhan--Ramanan \cite{narasimhan-ramanan}. In this sense, moduli spaces of twisted equivariant bundles are a generalization of Prym varieties.

Twisted equivariant bundles also appear in the study of principal bundles with non-connected structure group and \textbf{non-abelian Hodge theory} --- see our joint work with García-Prada--Gothen--Mundet i Riera \cite{BGGM}, and the work of García-Prada--Gothen--Mundet i Riera \cite{oscar-ignasi-gothen}.

Our results lay the groundwork for a deeper exploration of the geometry and topology of these moduli spaces, such as Harder--Narasimhan stratifications. An obvious complementary direction of study would be a stack-theoretic construction using \textbf{beyond GIT} techniques developed by Halpern--Leistner \cite{beyond-git} and others. Another very relevant direction of research is the generalization to twisted equivariant Higgs pairs, which are very interesting objects from the point of view of non-abelian Hodge theory and the geometric Langlands programme as noted above. Finally, by allowing for the image of $\theta$ to contain anti-holomorphic automorphisms of $G$, one could enter the realm of \textbf{pseudo-real $G$-bundles} introduced by Biswas, García-Prada, Huismann and Hurtubise in \cite {oscar-biswas-hurtubise} and \cite{BHH}. 

\section{Twisted equivariant \texorpdfstring{$G$}{G}-bundles}\label{section-twisted-equivariant-bundles}

Consider a compact Riemann surface $X$ and a connected reductive complex Lie group $G$ with center $Z$. 
Let $\Gamma$ be a finite group of automorphisms of $X$ acting on the right and take a homomorphism 
$$\theta:\Gamma\to\Aut(G);\,\gamma\mapsto\tg.$$
A \textbf{2-cocycle} of $\Gamma$ with values in $Z$ is a map $c:\Gamma\times\Gamma\to Z$ which satisfies $c(\gamma,1)=c(1,\gamma)=1$ and
\begin{equation}\label{eq-2-cocycle}
  \theta_{\gamma_0}(c(\gamma_1,\gamma_2))c(\gamma_0,\gamma_1\gamma_2)
  =c(\gamma_0,\gamma_1)c(\gamma_0\gamma_1,\gamma_2).
\end{equation} 
for every $\gamma,\gamma_0,\gamma_1,\gamma_2$ and $\gamma_3\in\Gamma$. 
Note that $\Int(G)$ acts trivially on $Z$, so (\ref{eq-2-cocycle}) only depends on the outer class of $\theta$. The set of 2-cocycles forms a group which we denote by $Z^2_{\theta}(\Gamma,Z)$. There is an action of the group of maps $s:\Gamma\to Z$ on $Z^2_{\theta}(\Gamma,Z)$, namely 
\begin{align*}
    Z^{\Gamma}\times Z^2_{\theta}(\Gamma,Z)\to Z^2_{\theta}(\Gamma,Z);\,(c,s)\mapsto s\cdot c:\Gamma\times\Gamma\to Z,\\
    s\cdot c(\gamma_1,\gamma_2)=c(\gamma_1,\gamma_2)s(\gamma_1)\theta_{\gamma_1}(s(\gamma_2))s(\gamma_1\gamma_2)^{-1}.
\end{align*}
The quotient of $Z^2_{\theta}(\Gamma,Z)$ by this $Z^{\Gamma}$-action is the \textbf{second cohomology group} $H^2_{\theta}(\Gamma,Z)$.
Fix a 2-cocycle $c\in Z^2_{\theta}(\Gamma,Z)$.

Given a holomorphic $G$-bundle $\bundle$ over $X$, a \textbf{right $(\theta,c)$-twisted $\Gamma$-equivariant action} on $\bundle$ is a holomorphic map
    \begin{equation}\label{eq-twisted-action-M}
        \bundle\times\Gamma\to \bundle;\,(q,\gamma)\mapsto p\bullet\gamma
    \end{equation}
such that the following hold.
\begin{itemize}
    \item For every $q\in \bundle$, $q\bullet 1=q$.
    \item For every $q\in \bundle$, $g\in G$ and $\gamma$ and $\gamma'\in\Gamma$,
    \begin{equation}\label{eq-twisted-equivariant-axioms}
        (q g)\bullet\gamma=(q\bullet\gamma)\theta_{\gamma}^{-1}(g)\andd(q\bullet\gamma)\bullet\gamma'=(qc(\gamma,\gamma'))\bullet(\gamma\gamma').
    \end{equation}
    
    \item The diagramme
      \begin{equation*}
          \begin{tikzcd}
            \bundle\arrow[r,"\bullet\gamma"]\arrow[d] & \bundle\arrow[d]\\
            X\arrow[r,"\cdot\gamma"] & X
          \end{tikzcd}
      \end{equation*}
      commutes for each $\gamma\in\Gamma$.
\end{itemize}

\begin{definition}
A \textbf{$(\theta,c)$-twisted $\Gamma$-equivariant $G$-bundle} over $X$ is a pair $(\bundle,\bullet)$ consisting of a holomorphic $G$-bundle $\bundle$ over $X$ and a $(\theta,c)$-twisted $\Gamma$-equivariant action $\bullet$ on $\bundle$.
Whenever there is no danger of confusion we may omit $\theta,c,\Gamma$ or $\bullet$ from the notation.
An \textbf{isomorphism} of $(\theta,c)$-twisted $\Gamma$-equivariant $G$-bundles is a $\Gamma$-equivariant isomorphism of $G$-bundles.
\end{definition}

\begin{remark}\label{remark-replace-theta}
    By \cite[Proposition 2.16]{BGGM}, the categories of $(\theta,c)$-twisted $\Gamma$-equivariant $G$-bundles and $(\theta',c')$-twisted $\Gamma$-equivariant $G$-bundles are equivalent whenever $\theta$ and $\theta'$ are in the same outer class, for some 2-cocycle $c'$ depending on $c$ and $\theta$. The 2-cocycle $c'$ is not unique, in fact we may set $\theta'=\theta$ and choose $c'$ to be any 2-cocycle cohomologous to $c$.
\end{remark}

\subsection{Semistable twisted equivariant \texorpdfstring{$G$}{G}-bundles}\label{section-semistable-twisted-equivariant-G-bundles}

We recall the stability notions for twisted equivariant $G$-bundles using the approach in \cite{oscar-ignasi-gothen}. By Remark \ref{remark-replace-theta}, after possibly replacing $\theta$ we may assume that there is a $\Gamma$-invariant maximal compact subgroup $K$ of $G$ with Lie algebra $\lie k$. Choose a $G$-invariant non-degenerate pairing $\langle\cdot,\cdot\rangle$ on $\lie g$. Every element $s\in i\lie k$ determines a parabolic subgroup $P_s$ with Lie algebra $\lie p_s$ of $G$, namely
\begin{equation}\label{eq-def-Ps}
    P_s:=\{g\in G\suhthat e^{ts}ge^{-ts}\,\text{remains bounded as}\;t\to\infty\}.
\end{equation}
If $L_s$ is its Levi subgroup, then $K_s:=K\cap L_s$ is a maximal compact subgroup of $L_s$ and the embedding $K_s\hookrightarrow P_s$ is a homotopy equivalence. Now let $\bundle$ be a $G$-bundle with a holomorphic reduction $\tau\in H^0(X,\bundle(G/P_s))$, where $\bundle(G/P_s)$ is the $G/P_s$-bundle associated to $E$ via the natural $G$-action; we denote by $\bundle_{\tau}:=\tau^*\bundle$ the corresponding $P_s$-bundle over $X$. Then there is a smooth reduction of structure group $\tau'\in\Omega^0(X,\bundle_{\tau}/K_s)$, and we may equip the corresponding $K_s$-bundle with a connection $A$ with curvature $F_A$. We define
\begin{equation}\label{eq-def-deg}
    \deg E(\tau,s):=\frac{i}{2\pi}\int_X\langle s,F_A\rangle.
\end{equation}

Note that, if $s$ is in the $\Gamma$-invariant part $i\lie k^{\Gamma}$ of $i\lie k$, then $P_s$ is $\Gamma$-invariant and so is $L_s$, since
\begin{equation}\label{eq-def-levi}
    L_s=\{g\in G\suhthat \lim_{t\to\infty} e^{ts}ge^{-ts}=g\}.
\end{equation}
In particular, $\theta$ induces $\Gamma$-actions on $G/P_s$ and $P_s/L_s$. If $\bundle$ is a $(\theta,c)$-twisted $\Gamma$-equivariant $G$-bundle then, combining the twisted equivariant action on $\bundle$ with the $\Gamma$-action on $G/P_s$, we obtain a $\Gamma$-equivariant structure on the fibre bundle $\bundle(G/P_s)$ by \cite[Proposition 3.16]{BGGM}. This lets us introduce the corresponding space of $\Gamma$-invariant reductions $H^0(X,\bundle(G/P_s))^{\Gamma}$. Given such a reduction $\tau\in H^0(X,\bundle(G/P_s))^{\Gamma}$, the corresponding $P_s$-bundle $\bundle_{\tau}$ inherits a $(\theta,c)$-twisted $\Gamma$-equivariant structure. Similarly,
the fibre bundle $\bundle(P_s/L_s)$ comes equipped with a $\Gamma$-equivariant action, with a corresponding space of $\Gamma$-invariant reductions $H^0(X,\bundle_{\tau}(P_s/L_s))^{\Gamma}$.


\begin{definition}\label{def-semistable-twisted-equivariant}
Let $\zk$ be the centre of a $\Gamma$-invariant maximal compact subalgebra $\lie k\subset\lie g$, and fix $\cha\in i\zk^{\Gamma}$. A $(\theta,c)$-twisted $\Gamma$-equivariant $G$-bundle $(\bundle,\bullet)$ over $X$ is:
\begin{itemize}
    \item \textbf{$\cha$-semistable} if $\deg E(\tau,s)\ge \pair{\cha}s$ for any $s\in i\lie k^{\Gamma}$ and any $\Gamma$-invariant reduction of structure group $\tau\in H^0(X,E(G/P_s))^{\Gamma}$.
    \item \textbf{$\cha$-stable} if $\deg E(\tau,s)> \pair{\cha}s$ for any $s\in i\lie k^{\Gamma}$ and any reduction of structure group $\tau\in H^0(X,E(G/P_s))^{\Gamma}$.
    \item \textbf{$\cha$-polystable} if it is semistable and, for every $s\in i\lie k^{\Gamma}$ and $\tau\in H^0(X,E(G/P_s))^{\Gamma}$ such that $\deg E(\tau,s)=\pair{\cha}s$, there is a further $\Gamma$-invariant holomorphic reduction of structure group $\tau'\in H^0(X,E_{\tau}(P_s/L_s))^{\Gamma}$.
\end{itemize}
If $\cha=0$, we omit it.
\end{definition}

\begin{remark}
If $\Gamma$ is the trivial group, we get the usual (poly, semi)stability notions for $G$-bundles.
\end{remark}

In Section \ref{section-bounded} we will need the following result, which is a special case of \cite[Proposition 4.1]{oscar-suratno}.

\begin{proposition}\label{prop-semistability-underlying-bundle}
A twisted equivariant $G$-bundle is $\cha$-(poly)semistable if and only if the underlying $G$-bundle is $\cha$-(poly)semistable.
\end{proposition}

\section{Twisted equivariant \texorpdfstring{$G$}{G}-bundles and pseudo-equivariant \texorpdfstring{$G$}{G}-bundles}\label{section-twisted-equivariant-and-pseudo-equivariant}

Pick a compact Riemann surface $X$, a finite group $\Gamma$ of complex automorphisms of $X$, a connected \textbf{semisimple} complex Lie group $G$ with center $Z$, a homomorphism $\theta:\Gamma\to\Aut(G)$ and a 2-cocycle $c\in Z^2_{\theta}(\Gamma,Z)$. The reason why we restrict to $G$ semisimple is that we first construct the moduli space under this assumption, and then extend our construction to arbitrary reductive $G$ in Section \ref{section-moduli-space-reductive}. In particular, since the centre $Z$ is discrete the stability parameter $\zeta=0$ and we omit it.

\subsection{Pseudo-equivariant \texorpdfstring{$G$}{G}-bundles}\label{section-pseudo-equivariant-bundles}

The \textbf{twisted product} $G\times_{(\theta,c)}\Gamma$ is the set $G\times\Gamma$ equipped with the group structure $(g,\gamma)(g',\gamma')=(g\tg(g')c(\gamma,\gamma'),\gamma\gamma')$ for every $(g,\gamma)$ and $(g',\gamma')$ in $G\times\Gamma$. 

Since $G\times_{(\theta,c)}\Gamma$ is a finite extension of $G$ and $G$ is reductive, the group $G\times_{(\theta,c)}\Gamma$ is reductive. Therefore, there exists an embedding $\iota:G\times_{(\theta,c)}\Gamma\hookrightarrow\GL(n,\C)$. 
Since $G$ is semisimple, $G=[G,G]$ and so $G$ has no non-trivial characters, which implies that $\iota(G)\subset \SL(n,\C)$. 
We fix an embedding $\iota:G\times_{(\theta,c)}\Gamma\hookrightarrow\SL(n,\C)$ once and for all. In particular every automorphism in $\theta(\Gamma)$ is realized as the restriction of an inner automorphism of $\SL(n,\C)$, namely $\tg=\Int_{(1,\gamma)}\vert_G$ for each $\gamma\in\Gamma$, where we write $(1,\gamma)$ instead of $\iota(1,\gamma)$ by abuse of notation. Here 
\begin{equation*}
    \Int_{(1,\gamma)}:\GL(n,\C)\to\GL(n,\C);\,g\mapsto(1,\gamma)g(1,\gamma)^{-1}
\end{equation*}
denotes the conjugation automorphism.

\begin{definition}
A \textbf{$\Gamma$-pseudo-equivariant $G$-bundle} of rank $n$ over $X$ is a triple $(E,h,f)$ consisting of a vector bundle $E$ of rank $n$, a reduction $h\in H^0(X,\iso(\oo_X^{\oplus n},E)/G)$ of structure group and a family of vector bundle isomorphisms $f=\{\fg:E\xrightarrow{\sim}\gamma^*E\}_{\gamma\in\Gamma}$ satisfying
\begin{equation}\label{eq-pseudo-equivariant-composition-f}
 f_1=\id,\; (\gamma^*f_{\gamma'})\circ f_{\gamma}=f_{\gamma\gamma'}
\end{equation}
and
\begin{equation}\label{eq-pseudo-equivariant-f-h}
    f_{\gamma}(h)=(\gamma^*h)(1,\gamma)^{-1}.
\end{equation}
for every $\gamma$ and $\gamma'\in\Gamma$. Here $\iso(\oo_X^{\oplus n},E)$ is the bundle of frames of $E$, and $f_\gamma(h)\in H^0(X,\iso(\oo_X^{\oplus n},\gamma^*E)/G)$ is the induced reduction of structure group. When there is no risk of confusion, we may omit $\Gamma$ from the notation. We sometimes call $f$ the \textbf{pseudo-equivariant action or structure}.

An \textbf{isomorphism} of pseudo-equivariant $G$-bundles $(E,h,f)\to(E',h',f')$ is an isomorphism of vector bundles $\phi:E\xrightarrow{\sim} E'$ such that $\phi(h)=h'$ and $\fg'\circ\phi=(\gamma^*\phi)\circ\fg$ for every $\gamma\in\Gamma$.
\end{definition}

\begin{remark}\label{remark-degree-0}
    For any pseudo-equivariant $G$-bundle $(E,h,f)$, the underlying vector bundle $E$ has trivial determinant, and in particular degree $0$. This is due to the fact that the bundle of frames of $E$ has a reduction of structure group to $G\subset SL(n,\C)$.
\end{remark}

\subsection{Twisted equivariant \texorpdfstring{$G$}{G}-bundles as pseudo-equivariant \texorpdfstring{$G$}{G}-bundles}\label{section-twisted-equivariant-as-pseudo-equivariant}

\begin{proposition}\label{prop-equivalence-of-categories-twisted-pseudo}
The following categories are equivalent:
\begin{enumerate}
    \item The category $\mathcal A$ of twisted equivariant $G$-bundles over $X$.
    \item The category $\mathcal B$ of pseudo-equivariant $G$-bundles over $X$.
\end{enumerate}
\end{proposition}
\begin{proof}
We construct an invertible functor $\mathcal A\to\mathcal B$. Let $\bundle_G$ be a $G$-bundle equipped with a twisted equivariant $\Gamma$-action. By (\ref{eq-twisted-equivariant-axioms}) this action is equivalent to a set of $G$-bundle isomorphisms $\fg':\bundle_G\to\gamma^*\tg(\bundle_G)$ satisfying
\begin{equation}\label{eq-twisted-equivariant-vs-pseudo}
    f'_1=1\andd(\gamma^*f'_{\gamma'})\circ f'_{\gamma}=f'_{\gamma\gamma'}c(\gamma,\gamma')
\end{equation}
for each $\gamma$ and $\gamma'\in\Gamma$, where we identify elements of $\GL(n,\C)$ with maps via the principal bundle action. 

Let $\bundle:=\bundle_G(\GL(n,\C))$ be the extension of structure group and let $h\in H^0(X,\bundle/G)$ be the corresponding reduction of structure group. Recall that $\Int_{(1,\gamma)}(\bundle)$ may be defined as the $\GL(n,\C)$-bundle that has the same total space as $\bundle$, but $\GL(n,\C)$-action twisted by $\Int_{(1,\gamma)}^{-1}$.  The set of isomorphisms $(\fg')_{\gamma\in\Gamma}$ induces a set of isomorphisms $\fg:\bundle\to \gamma^*\Int_{(1,\gamma)}(\bundle)$ satisfying (\ref{eq-twisted-equivariant-vs-pseudo}) and $\fg(h)=\gamma^*h$ --- here $h$ is regarded as a section of $\Int_{(1,\gamma)}(\bundle)$ via the identification of its underlying total space with $\bundle$. 

Multiplication by $(1,\gamma)^{-1}$ --- via the $\GL(n,\C)$-action on $\bundle$ --- provides an isomorphism $\gamma^*\Int_{(1,\gamma)}(\bundle)\to \gamma^*(\bundle)$, hence after composing we get a set of $\GL(n,\C)$-bundle isomorphisms
\begin{equation*}
    \fg:=(1,\gamma)^{-1}\gamma^*f_{\gamma}:\bundle\xrightarrow{\fg'}\gamma^*\Int_{(1,\gamma)}(\bundle)\xrightarrow{(1,\gamma)^{-1}}\gamma^*(\bundle)
\end{equation*}
sending $h$ to $\gamma^*h(1,\gamma)^{-1}$.
Moreover, (\ref{eq-twisted-equivariant-vs-pseudo}) implies (\ref{eq-pseudo-equivariant-composition-f}) since, for every $q\in \bundle$,
\begin{align*}
    (\gamma^*f_{\gamma'})\circ f_{\gamma}(q)&=(1,\gamma')^{-1}\gamma^*f'_{\gamma'}\circ (1,\gamma)^{-1}f'_{\gamma}(q)
    \\&=
    \gamma^*f'_{\gamma'}(\fg'(q)(1,\gamma)^{-1})(1,\gamma')^{-1}
    \\&=
    \left((\gamma^*f'_{\gamma'})\circ\fg'(q)\right)
    \Int_{(1,\gamma')}^{-1}((1,\gamma)^{-1})(1,\gamma')^{-1}
    \\&=
    f'_{\gamma\gamma'}(qc(\gamma,\gamma'))
    \left((1,\gamma')\Int_{(1,\gamma')}^{-1}((1,\gamma))\right)^{-1}
    \\&=
    f'_{\gamma\gamma'}(q)\Int_{(1,\gamma\gamma')}^{-1}(c(\gamma,\gamma'))
    \left((1,\gamma)(1,\gamma')\right)^{-1}
    \\&=
    f'_{\gamma\gamma'}(q)
    (1,\gamma\gamma')^{-1}c(\gamma,\gamma')(1,\gamma\gamma')
    \left((c(\gamma,\gamma'),\gamma\gamma')\right)^{-1}
    \\&=
    f'_{\gamma\gamma'}(q)
    (1,\gamma\gamma')^{-1}c(\gamma,\gamma')(1,\gamma\gamma')
    (1,\gamma\gamma')^{-1}
    c(\gamma,\gamma')^{-1}
    \\&=
    f'_{\gamma\gamma'}(q)
    (1,\gamma\gamma')^{-1}
    \\&=
    f_{\gamma\gamma'}(q).
\end{align*}

Let $E:=\bundle(\C^n)=\bundle\times_{\GL(n,\C)}\C^n$ be the associated vector bundle of rank $n$. Then the isomorphisms $\fg$ induce isomorphisms $\fg:E\xrightarrow{\sim}\gamma^*E$, and $(E,h,f=\{\fg\}_{\gamma\in\Gamma})$ is a pseudo-equivariant $G$-bundle. 
Therefore, starting with a twisted equivariant \texorpdfstring{$G$}{G}-bundle $\bundle_G$, we have obtained a pseudo-equivariant $G$-bundle $(E,h,f)$. Furthermore, an isomorphism of twisted equivariant $G$-bundles induces, by extension of structure group, an isomorphism of pseudo-equivariant $G$-bundles. Working backwards we see that this functor is invertible.
\end{proof}

\begin{definition}\label{def-semistable-pseudo-equivariant-bundle}
A pseudo-equivariant $G$-bundle is called \textbf{(semi)stable} if the corresponding twisted equivariant $G$-bundle is so.
\end{definition}

\section{Pseudo-equivariant \texorpdfstring{$G$}{G}-bundles as pseudo-equivariant \texorpdfstring{$\rho$}{rho}-pairs}\label{section-pseudo-equivariant-bundles-as-pseudo-equivariant-pairs}

Following the identification of $G$-bundles with swamps in \cite[Section 2.4]{schmitt}, we identify pseudo-equivariant $G$-bundles with analogous objects. We still assume that $G$ is a connected semisimple complex Lie group in this section. As in Section \ref{section-pseudo-equivariant-bundles}, fix the data $X$, $\Gamma$, $\theta$, $c$ and $\iota:G\times_{(\theta,c)}\Gamma\hookrightarrow\SL(n,\C)$.

\subsection{Semistable pseudo-equivariant pairs}\label{section-semistable-pseudo-equivariant-pairs}
Recall that a representation $\rho:\GL(n,\C)\to\GL(W)$ is \textbf{homogeneous} if $\rho(z\id)=z^k\id$ for some integer $k$, which is called the \textbf{degree} of $\rho$, or $k=\deg\rho$.
Let $\rho:\GL(n,\C)\to\GL(W)$ be a homogeneous representation. Assume that $W$ is also equipped with a right linear $\GL(n,\C)$-action commuting with $\rho$, denoted by
\begin{equation}\label{eq-right-GL-action}
    W\times\GL(n,\C)\to W;\,(w,g)\mapsto w\cdot g.
\end{equation}
Given a vector bundle $E\to X$ of rank $n$, denote by
\begin{equation*}
    E(W):=\iso(\oo_X^{\oplus n},E)\times_{\rho} W
\end{equation*} 
the associated vector bundle with fibre $W$.
Since the right $\GL(n,\C)$-action (\ref{eq-right-GL-action}) commutes with $\rho$, $E(W)$ is naturally equipped with a right $\GL(n,\C)$-action.

\begin{definition}\label{def-pseudo-equivariant-pair}
A \textbf{$\Gamma$-pseudo-equivariant $\rho$-pair} --- abbreviated pseudo-equivariant pair when $\Gamma$ and $\rho$ are understood from context --- is a triple $(E,\sigma,f)$ consisting of a vector bundle $E$ of rank $n$, a non-zero homomorphism $\sigma:E(W)\to \oo_X$ and a family of vector bundle isomorphisms $f=\{\fg:E\xrightarrow{\sim}\gamma^*E\}_{\gamma\in\Gamma}$ satisfying
(\ref{eq-pseudo-equivariant-composition-f}) --- namely
\begin{equation*}
 f_1=\id\andd (\gamma^*f_{\gamma'})\circ f_{\gamma}=f_{\gamma\gamma'}
\end{equation*}
--- and
\begin{equation}\label{eq-pseudo-equivariant-pairs-f-sigma}
    \fg(\sigma)=\gamma^*\sigma\cdot(1,\gamma)^{-1}.
\end{equation}

An \textbf{isomorphism} of pseudo-equivariant pairs $(E,\sigma,f)\xrightarrow{\sim}(E',\sigma',f')$ consists of an isomorphism of vector bundles $\phi:E\xrightarrow{\sim}E'$ such that $(\gamma^*\phi)\circ \fg=\fg'\circ\phi$ for each $\gamma\in\Gamma$ and $\phi(\sigma)=\sigma'$.
\end{definition}

In order to define the stability notions for pseudo-equivariant $\rho$-pairs, we introduce some notation.

\begin{definition}
Consider a vector space $V$ over a field $k$ and a representation $\trho:\GL(n,k)\to\GL(V)$. A \textbf{1-parameter subgroup} of $\GL(n,k)$ is a group homomorphism $\lambda:\mathbb G_m(k)\to \GL(n,k)$, where $\mathbb G_m(k)$ is the multiplicative group of $k$. Since $\mathbb G_m(k)$ is reductive, the representation $\trho\circ \lambda$ is completely reducible, hence we can decompose $V$ as a sum of weight spaces $\bigoplus_{j=1}^JV^{\beta_j}$, where $t\in \mathbb G_m(k)$ acts on $V^{\beta_j}$ by multiplication by $t^{\beta_j}$ and $\beta_j\in\Z$. Given $v\in V$, we may decompose $v=\sum_{j=1}^Jv^{\beta_j}$ uniquely, where $v^{\beta_j}\in V^{\beta_j}$, and define
\begin{equation}\label{eq-def-mu}
    \mu(v,\lambda):=-\min\{\beta_j\suhthat v^{\beta_j}\ne 0\}.
\end{equation}
\end{definition}

\begin{remark}\label{remark-properties-mu}
    It is straightforward to check that, given two representations $\trho_i:\GL(n,k)\to\GL(V_i)$, two vectors $v_i\in V_i$ ($i=1,2$) and a 1-parameter subgroup $\lambda:\Gm(k)\to\GL(n,k)$,
    \begin{equation*}
        \mu(v_1\otimes v_2,\lambda)=\mu(v_1,\lambda)+\mu(v_2,\lambda)\andd
        \mu(v_1\oplus v_2,\lambda)=\max\left(\mu(v_1,\lambda),\mu(v_2,\lambda)\right).
    \end{equation*}
\end{remark}

Fix a pseudo-equivariant pair $(E,\sigma,f)$ and let 
\begin{equation*}
    E_\bullet:\{0\}\subsetneq E_1\subsetneq\dots\subsetneq E_J\subsetneq E
\end{equation*}
be a filtration of $E$ and let $\alpha_\bullet=(\alpha_1,\dots,\alpha_J)\in\Q_{>0}^J$. We call the pair $(E_\bullet,\alb)$ a \textbf{weighted filtration} of $E$, and we call it \textbf{proper} if $E_\bullet$ is proper. Then we set
\begin{equation}
    M(\eb,\alb):=\sum_{j=1}^J\alpha_j(\deg(E)\rk(E_j)-\deg(E_j)\rk(E)).
\end{equation}


Set $\hhom_\Gamma(E):=\bigoplus_{\gamma\in\Gamma}\hhom(E,\gamma^*E)$. Let $\eta\in X$ be the generic point of $X$ when thought of as an algebraic curve. The restrictions $\mathbb E:=E\vert_\eta$, $\mathbb W:=E(W)\vert_\eta$ and $\mathbb H_\Gamma:=\hhom_\Gamma(E)\vert_\eta$ are vector spaces over the function field $\C(X)$ of $X$. The restriction of the filtration $\eb\vert_\eta$ is a filtration $\mathbb \eb$ of $\mathbb E$ by vector subspaces. Let $r\in\N$ be such that $r\alb\in \frac 1n\Z^J$.
Let $\lambda:\mathbb G_m(\C(X))\to \GL(\mathbb E)\cong\GL(n,\C(X))$ be a \textbf{1-parameter subgroup} whose weight space decomposition is $\mathbb E=\bigoplus_{j=1}^J\mathbb E^{r\alpha_j}$ and satisfies $\mathbb E_j=\bigoplus_{i=1}^{j}\mathbb E^{r\alpha_i}$ --- $\lambda$ is uniquely defined up to the conjugation action of the parabolic subgroup $\stab_{\GL(\mathbb E)}(\Eb)\le\GL(\mathbb E)$. Set $f^{*-1}:=\{\fg^{*-1}:E^*\to\gamma^*E^*\}_{\gamma\in\Gamma}$,
\begin{equation}\label{eq-def-mus}
    \mu(\eb,\alb,\sigma):=\frac 1{r^{\deg\rho}}\mu(\sigma\vert_\eta,\lambda)\andd \mu(\eb,\alb,f):=\frac 1{r}\max(\mu(\oplus f\vert_\eta,\lambda),\mu(\oplus f^{*-1}\vert_\eta,\lambda)),
\end{equation}
where $\sigma\vert_\eta\in\mathbb W$, and  $\oplus f\vert_\eta$ and $\oplus f^{*-1}\vert_\eta\in\mathbb H_\Gamma$ --- here we use the notation $\oplus f:=\bigoplus_{\gamma\in{\Gamma}}\fg$.
It is straightforward to check that (\ref{eq-def-mus}) does not depend on $r$.
By \cite[Proposition 1.5.1.35]{schmitt} $\mu$ is invariant by the $\stab_{\GL(n,\C(X))}(\Eb)\le\GL(n,\C(X))$-action on $\lambda$, therefore (\ref{eq-def-mus}) only depends on the data $(E_\bullet,\alb,\sigma,f)$.

\begin{definition}\label{def-delta,chi-semistable-pairs}
Given $\delta,\chi\in\Q_{>0}$, a pseudo-equivariant pair $(E,\sigma,f)$ is \textbf{$(\delta,\chi)$-(semi)stable} if, for every proper weighted filtration $(\eb,\alb)$ of $E$,
\begin{equation}\label{eq-delta,chi-semistable}
    M(\eb,\alb)+\delta\mu(\eb,\alb,\sigma)+\chi\mu(\eb,\alb,f)\geqp
0.
\end{equation}
\end{definition}

\begin{remark}\label{remark-pairs-vs-swamps}
    Given a pseudo-equivariant $\rho$-pair $(E,\sigma,f)$, the underlying pair $(E,\sigma)$ defines a $\rho$-swamp $(E,\oo_X,\sigma)$ in the sense of \cite{schmitt}. Moreover, the swamp $(E,\oo_X,\sigma)$ is $\delta$-(semi)stable, in the sense of \cite[Section 2.3.2]{schmitt}, if and only if $(E,\sigma,f)$ is $(\delta,0)$-(semi)stable.
\end{remark}

The following technical lemma will be useful when working with the notion of $(\delta,\chi)$-(semi)stability.

\begin{lemma}\label{lemma-delta,chi-semistable-finite-set}
    There exists a finite set
    \begin{equation*}
        S\subset \{(n_{\bullet},\alb)\suhthat n_j\in\N,\alpha_j\in\Q_{\ge 0},\sum_jn_j=n\},
    \end{equation*}
    independent of $\delta$ and $\chi$, such that such that, for any pseudo-equivariant pair $(E,\sigma,f)$, the following two conditions are equivalent.
    \begin{enumerate}
        \item $(E,\sigma,f)$ is $(\delta,\chi)$-(semi)stable.
        \item For any weighted filtration $(\eb,\alb)$ such that $(\rk \eb:=(\rk E_j)_j,\alb)\in S$, (\ref{eq-delta,chi-semistable}) holds, namely
        \begin{equation*}
            M(\eb,\alb)+\delta\mu(\eb,\alb,\sigma)+\chi\mu(\eb,\alb,f)\geqp
            0.
        \end{equation*}
    \end{enumerate}
\end{lemma}

\begin{remark}
    Lemma \ref{lemma-delta,chi-semistable-finite-set} may be summarized with the slogan \textit{$(\delta,\chi)$-semistability only has to be checked for weighted filtrations $(\eb,\alb)$ such that $(\rk \eb,\alb)\in S$.}
\end{remark}

\begin{proof}
    We follow the argument in \cite[Example 1.5.1.18]{schmitt}. Choose a maximal torus $T\subset\GL(n,\C(X))$ with Lie algebra $\lie t$ and a Weyl chamber $K\subset X(\lie t)$, where $X(\lie t)$ is the vector space of characters of $\lie t$. Given a representation $\rho:\GLX\to \GL(V)$, let $X(V)\subset X(\lie t)$ be the set of weights of $\rho$. Given a subset $R\subset X(V)$ and $\chi\in R$, set
    \begin{equation*}
        K(V,R,\chi):=\{k\in K\suhthat \langle k,\chi\rangle\le \langle k,\chi'\rangle\;\forall \chi'\in R\}\subset X(\lie t).
    \end{equation*}
    This is a rational polyhedral cone in $\lie t$, and 
    \begin{equation*}
        K=\bigcup_{\chi\in R}K(V,R,\chi)
    \end{equation*}
    is a fan decomposition of $K$.

    Let
    \begin{equation*}
        V_1:=\C(X)^n,\,V_2:=V\otimes \C(X)
    \end{equation*}
    and
    \begin{equation*}
        V_3:=\bigoplus_{\gamma\in\Gamma}
        \Hom(\C(X),\gamma^*\C(X))
        \oplus
        \Hom(\C(X)^*,\gamma^*\C(X)^*),
    \end{equation*}
    equipped with the represenations $\GLX\to\GL(V_i)$ induced by the standard representation of $\GLX$ --- of course $\gamma^*\C(X)^*\cong\gamma^*\C(X)\cong\C(X)$ for every $\gamma$, but we are trying to use suggestive notation here. For every choice of subsets $R_i\subset X(V_i)$, the expression
    \begin{equation*}
        K=\bigcup_{\chi_i\in R_i}\bigcap_{i=1,2,3}K(V_i,R_i,\chi_i)
    \end{equation*}
    is a fan decomposition of $K$. Let $\Lambda((V_i)_{i=1,2,3},(R_i)_{i=1,2,3},(\chi_i)_{i=1,2,3})$ be the set of edges of $\bigcap_{i=1,2,3}K(V_i,R_i,\chi_i)$, and let
    \begin{equation*}
        \Lambda:=\bigcup_{R_i\subset X(V_i),\chi_i\in R_i}\Lambda((V_i),(R_i),(\chi_i)).
    \end{equation*}
    
    \textbf{Claim:} Fix an isomorphism $X(\lie t)\cong X^*(\lie t)$ between the spaces of characters and cocharacters of $\lie t$. Then, for every $v_i\in V_i$, 
    \begin{equation}\label{eq-mu-delta-chi-wi}
        \mu(v_1,\lambda)+\delta\mu(v_2,\lambda)+\chi\mu(v_3,\lambda)\geqp0
    \end{equation}
    for every non-trivial 1-parameter $\lambda$ of $\GLX$, if and only if
    \begin{equation}\label{eq-mu-delta-chi-gwi}
        \mu(g\cdot v_1,\lambda)+\delta\mu(g\cdot v_2,\lambda)+\chi\mu(g\cdot v_3,\lambda)\geqp0
    \end{equation}
    for every $g\in\GLX$ and every $\lambda$ corresponding to an element of $\Lambda$.

    Since $\mu(g\cdot v_i,\lambda)=\mu(v_i,g^{-1}\lambda g)$, (\ref{eq-mu-delta-chi-wi}) implies \ref{eq-mu-delta-chi-gwi}. To prove the converse, assume (\ref{eq-mu-delta-chi-gwi}) and pick $v_i\in V_i$. Let $\lambda:\Gm(\C(X))\to\GLX$ be a non-trivial 1-parameter subgroup. There exists $g\in\GLX$ such that $g\lambda g^{-1}$ lands in $T$. We may associate to $g\lambda g^{-1}$ a character $k$ which, after conjugating $\lambda$ further, we may assume to be in $K$. Let $R_i\subset X(V_i)$ be the subset of weights appearing in the weight decomposition of $g\cdot v_i$. By (\ref{eq-def-mu}),
    \begin{equation*}
        \mu(v_i,\lambda)=\mu(g\cdot v_i,g\lambda g^{-1})=-\langle k,\chi_i\rangle,
    \end{equation*}
    where $\chi_i\in X(V_i)$ satisfies that $\langle k,\chi_i\rangle\le \langle k,\chi_i'\rangle$ for every $\chi_i'\in R_i$. In other words, $k\in K(V_i,R_i,\chi_i)$ for $i=1,2,3$, or $k\in \bigcap_iK(V_i,R_i,\chi_i)$. There is a decomposition
    \begin{equation*}
        k=\sum_{k'\in \Lambda((V_i)_i,(R_i)_i,(\chi_i)_i)}l_{k'}^kk',
    \end{equation*}
    for some positive numbers $l_{k'}^k$. 
    Then,
    \begin{equation*}
        \mu(v_i,\lambda)=-\sum_{k'\in \Lambda((V_i)_i,(R_i)_i,(\chi_i)_i)}l_{k'}^k\langle k',\chi_i\rangle.
    \end{equation*}
    Note that, if we regard $k'$ as a 1-parameter subgroup by abuse of notation, via the chosen isomorphism $X(\lie t)\cong X^*(\lie t)$, then $\mu(g\cdot v_i,k')=-\langle k',\chi_i\rangle$, since $k'\in K(V_i,R_i,\chi_i)$. Therefore,
    \begin{equation*}
        \mu(v_i,\lambda)=\sum_{k'\in \Lambda((V_i)_i,(R_i)_i,(\chi_i)_i)}l_{k'}^k\mu(g\cdot v_i,k'),
    \end{equation*}
    and so
    \begin{align*}
        \mu(v_1,\lambda)+\delta\mu(v_2,\lambda)+\chi\mu(v_3,\lambda)=\\
        \sum_{k'\in \Lambda((V_i)_i,(R_i)_i,(\chi_i)_i)}l_{k'}^k
        \left(
        \mu(g\cdot v_1,k')+\delta\mu(g\cdot v_2,k')+\chi\mu(g\cdot v_3,k')
        \right)\geqp0,
    \end{align*}
    where the last inequality follows from $\Lambda((V_i)_i,(R_i)_i,(\chi_i)_i)\subset \Lambda$ and $l_{k'}^k\ge 0$, as required.
    
    To finish the proof, we identify $K$ with $\R^n_{\ge 0}$, with basis $(e_1,\dots,e_n)$. For each $k\in \Lambda$, let $\beta_j^k\ge 0$ be such that
    \begin{equation*}
        k=\sum_j\beta^k_je_j.
    \end{equation*}
    Consider the set of tuples $\{(\beta^k_j)_j\}\subset \R^n_{\ge 0}$. There is a map
    \begin{equation*}
        \{(\beta^k_j)_j\}\to\{(n_{\bullet},\alb)\suhthat n_j\in\N,\alpha_j\in\Q_{\ge 0},\sum_jn_j=n\},
    \end{equation*}
    given by ordering each tuple $(\beta^k_j)_j$, partitioning the tuple into sets of equal numbers and letting $n_{\bullet}$ being the data of the partition and $\alb$ represent the value of $(\beta^k_j)$ on each block. Denoting the image by $S$, the lemma follows from the claim.
\end{proof}

\subsection{Pseudo-equivariant \texorpdfstring{$G$}{G}-bundles as pseudo-equivariant pairs}\label{section-pseudo-equivariant-bundles-as-pairs}

Let $(E,h,f)$ be a pseudo-equivariant $G$-bundle. The commutative diagramme
\begin{equation*}
    \begin{tikzcd}
    \iso(\oo_X^{\oplus n},E)\arrow{r}\arrow{d}&\hhom(\oo_X^{\oplus n},E)\arrow{d}\\
    \iso(\oo_X^{\oplus n},E)/G\arrow{r}&\hhom(\oo_X^{\oplus n},E)\sslash G,
    \end{tikzcd}
\end{equation*}
where the double slash denotes GIT quotient, implies that the total space of $\iso(\oo_X^{\oplus n},E)/G$ is an open subvariety of the total space of $\hhom(\oo_X^{\oplus n},E)\sslash G$. Indeed, since $G$ is semisimple, $\iota(G)\subset\SL(n,\C)$, hence the determinant morphism $\det:\hhom(\oo_X^{\oplus n},E)\to \oo_X$
is $G$-equivariant and descends to a morphism
\begin{equation*}
    \det:\hhom(\oo_X^{\oplus n},E)\sslash G\to \oo_X.
\end{equation*}
Moreover, $\iso(\oo_X^{\oplus n},E)/G=\det^{-1}(\oo_X^*)$, where $\oo_X^*\subset\oo_X$ denotes the subsheaf of functions with no zeros, and 
\begin{equation}\label{eq-iso-vs-det-1(0)}
    H^0(X,\iso(\oo_X^{\oplus n},E)/G)=H^0(X,\hhom(\oo_X^{\oplus n},E)\sslash G)\setminus \det{}^{-1}(0),
\end{equation}
where $0$ denotes the zero section of $\oo_X$.

On the other hand $\hhom(\oo_X^{\oplus n},E)\sslash G=\spec(\sym^*(\oo_X^{\oplus n}\otimes E^*)^G)$, where $\sym^*(\oo_X^{\oplus n}\otimes E^*)^G\subset\sym^*(\oo_X^{\oplus n}\otimes E^*)$ is the subsheaf of $G$-invariant sections, hence a global section of the sheaf $\hhom(\oo_X^{\oplus n},E)\sslash G$ is just a homomorphism of sheaves $\sym^*(\oo_X^{\oplus n}\otimes E^*)^G\to\oo_X$. Note that $\sym^*(\oo_X^{\oplus n}\otimes E^*)^G$ is finitely generated over $\oo_X$, therefore, by \cite[Chapter III, Section 8, Lemma]{mumford}, there exists a positive integer $d$ such that $\sym^{*d}(\oo_X^{\oplus n}\otimes E^*)^G$ is generated by its degree 1 part, that is $\sym^{d}(\oo_X^{\oplus n}\otimes E^*)^G$. Via the Veronese embedding, $\sym^{*}(\oo_X^{\oplus n}\otimes E^*)^G\cong \sym^{*d}(\oo_X^{\oplus n}\otimes E^*)^G$. Therefore, a homomorphism $h:\sym^*(\oo_X^{\oplus n}\otimes E^*)^G\to\oo_X$ is completely determined by a homomorphism
\begin{equation}\label{eq-h-as-hom-sym-to-OX}
    h:\sym^{d}(\oo_X^{\oplus n}\otimes E^*)^G\to\oo_X
\end{equation}
which we also denote by $h$ by abuse of notation.

Next we want to express $h$ as a homomorphism $E(W)\to\oo_X$, where $W$ satisfies the assumptions of Section \ref{section-pseudo-equivariant-bundles}. First we need a technical lemma. Given positive integers $a,b$ and $c$, set
\begin{equation*}
    \Cabc:=\left(\left(\C^n\right)^{\otimes a}\right)^{\oplus b}\otimes(\bigwedge^n\C^n)^{\otimes-c}.
\end{equation*}

\begin{lemma}\label{lemma-homogeneous-rep}    
    Let $\rho:\GL(n,\C)\times\GL(n,\C)\to\GL(W)$ be a representation that is homogeneous on the first factor, in the sense that the composition
    \begin{equation*}
        \rho_1:\GL(n,\C)\xrightarrow{\id\times 1}\GL(n,\C)\times\GL(n,\C)\xrightarrow{\rho}\GL(W)
    \end{equation*}
    is homogeneous. Then there exist $a,b$ and $c$ such that $W$ is a direct summand of $\Cabc$ and $\rho$ extends to a representation
    \begin{equation*}
        \trho:\GL(n,\C)\times\GL(n,\C)\to\GL(\Cabc)
    \end{equation*}
    such that the composition
    \begin{equation*}
        \trho_1:\GL(n,\C)\xrightarrow{\id\times 1}\GL(n,\C)\times\GL(n,\C)\xrightarrow{\trho}\GL(\Cabc)
    \end{equation*}
    is induced by the standard representation of $\GL(n,\C)$.
\end{lemma}
\begin{proof}
    This proof is modelled on \cite[proof of 5.3, Proposition]{kraft-representations}. Let $\M n$ be the ring of $(n\times n)$-matrices. Picking a basis of $W$, we may identify $\rho$ with a $(k\times k)$-matrix of regular functions in $\C[\GL(n,\C)\times \GL(n,\C)]=\C[\M n\times\M n][\pi_1\circ\det^{-1},\pi_2\circ\det^{-1}]$, where $k=\dim W$ and $\pi_i:\GLL\to\GL(n,\C)$ is the projection onto the $i$-th factor. 
    After replacing $\rho$ with $(\pi_1\circ\det)^{\otimes c}\otimes(\pi_2\circ\det)^{\otimes c}\otimes \rho$ for some $c$ big enough, we may assume that all these regular functions are in $\C[\MM n]$, so $\rho$ extends to a multiplicative map
    \begin{equation*}
        \rho:\MM n\to\M k.
    \end{equation*}
    The left $\GLL$-action on $\MM n$ given by $g*A:=A\cdot g^{-1}$ provides $\C[\MM n]$ with the $\GLL$-module structure
    \begin{equation*}
        (g* f)(A):=f(g^{-1}* A),
    \end{equation*}
    where $g\in\GLL$, $f\in\C[\MM n]$ and $A\in\MM n$. Given $\nu\in W^*$, define the linear map
    \begin{equation*}
        \phi_\nu:W\to\C[\MM n];\,w\mapsto\left(A\mapsto \nu(\rho(A)(w))\right).
    \end{equation*}
    It is straightforward to check that $\phi_\lambda$ is a map of $\GLL$-modules. Choosing a basis $\nu_1,\dots,\nu_k$ of $W^*$, we obtain an injective map of $\GLL$-modules
    \begin{equation}\label{eq-lemma-rep-W-in-C[]}
        \phi:W\hookrightarrow\C[\MM n]^{\oplus k};\,w\mapsto(\phi_{\nu_1}(w),\dots,\phi_{\nu_k}(w)).
    \end{equation}
    Now note that
    \begin{align*}
        \C[\MM n]
        &=
        \C[\M n]\otimes\C[\M n]
        \\&=
        \Sym^*\left(\M n^*\right)\otimes
        \Sym^*\left(\M n^*\right)
        \\&\cong
        \Sym^*\left((\C^n)^{\oplus n*}\right)\otimes
        \Sym^*\left((\C^n)^{\oplus n*}\right)
        \\&\subset 
        \bigoplus_{d_1,d_2\ge 0}
        \left((\C^n)^{\oplus n*}\right)^{\otimes d_1}\otimes
        \left((\C^n)^{\oplus n*}\right)^{\otimes d_2}
    \end{align*}
    as $\GLL$-modules. Combining this with (\ref{eq-lemma-rep-W-in-C[]}), and using the fact that $W$ has finite dimension, we find an injective map of $\GLL$-modules 
    \begin{equation*}
        W\hookrightarrow\left(\bigoplus_{d\ge d_1,d_2\ge 0}
        \left((\C^n)^{n*}\right)^{\otimes d_1}\otimes
        \left((\C^n)^{n*}\right)^{\otimes d_2}\right)^{\oplus k},
    \end{equation*}
    for a big enough positive integer $d$. Since $\rho_1$ is homogeneous, $W$ is the degree $d_1$ part for some $d_1$, i.e.
    \begin{equation*}
        W\hookrightarrow\left(\bigoplus_{d\ge d_2\ge 0}
        \left((\C^n)^{n*}\right)^{\otimes d_1}\otimes
        \left((\C^n)^{n*}\right)^{\otimes d_2}\right)^{\oplus k}.
    \end{equation*}
    Moreover, the composition 
    \begin{equation*}
        \GL(n,\C)\xrightarrow{\id\times 1}\GL(n,\C)\times\GL(n,\C)\to\GL
        \left(\left(\bigoplus_{d\ge d_2\ge 0}
        \left((\C^n)^{n*}\right)^{\otimes d_1}\otimes
        \left((\C^n)^{n*}\right)^{\otimes d_2}\right)^{\oplus k}
        \right)
    \end{equation*}
    is isomorphic to the representation
    \begin{equation*}
        \GL(n,\C)\to \GL\left(\left((\C^n)^{\otimes a}\right)^{\oplus b}\right)
    \end{equation*}
    induced by the standard representation, for some positive integers $a$ and $b$. The proof is finished if we note that $\GLL$-submodules and direct summands are the same, since $\GLL$ is reductive.
\end{proof}

Consider the left linear $\GLL$-action on $\C^n\otimes \C^{n*}$, given by
\begin{equation*}
    \GLL\times(\C^n\otimes \C^{n*})\to(\C^n\otimes \C^{n*});\, (g_1,g_2)\cdot(v,\nu)=(g_2\cdot v,g_1\cdot\nu),
\end{equation*}
where the $\GL(n,\C)$-actions on $\C^n$ and $\C^{n*}$ are the standard representation and the dual of the standard representation, respectively. This is homogeneous on both $\GL(n,\C)$-factors. It induces a representation
\begin{equation*}
    \rho:\GLL\to\GL\left(\sym^{d}(\C^n\otimes \C^{n*})^G\right)
\end{equation*} which is also homogeneous on both factors, in the sense that the compositions
\begin{equation*}
    \rho_1:\GL(n,\C)\xrightarrow{\id\times 1}\GL(n,\C)\times\GL(n,\C)\xrightarrow{\rho}\GL\left(\sym^{d}(\C^n\otimes \C^{n*})^G\right)
\end{equation*}
and
\begin{equation*}
    \rho_2:\GL(n,\C)\xrightarrow{1\times \id}\GL(n,\C)\times\GL(n,\C)\xrightarrow{\rho}\GL\left(\sym^{d}(\C^n\otimes \C^{n*})^G\right)
\end{equation*}
are both homogeneous.
By Lemma \ref{lemma-homogeneous-rep}, there exist integers $a,b$ and $c$ and a surjective homomorphism of $\GL(n,\C)$-representations
\begin{equation*}
    \phi_{a,b,c}:\C^n_{a,b,c}:=\left(\left(\C^n\right)^{\otimes a}\right)^{\oplus b}\otimes(\bigwedge^n\C^n)^{\otimes-c}\to \Sym^{d}(\C^n\otimes \C^{n*})^G,
\end{equation*}
where $\GL(n,\C)$ acts on $\Sym^{d}(\C^n\otimes \C^{n*})^G$ via $\rho_1$, and such that $\rho_2$ extends to a linear $\GL(n,\C)$-action on $\Cabc$ commuting with the $\GL(n,\C)$-action induced by the standard representation. Setting $W=\Cabc$, equipped with the representation $\rhoabc:\GL(n,\C)\to\GL(\Cabc)$ induced by the standard one and the right action
\begin{equation*}
    \Cabc\times\GL(n,\C)\to\Cabc;\,(v,g)\mapsto v\cdot g:= \rho_2(g)^{-1}v,
\end{equation*}
we have the necessary ingredients to make sense of Definition \ref{def-pseudo-equivariant-pair}.

The surjection $\phi_{a,b,c}$ induces a surjective homomorphism of sheaves
\begin{equation}\label{eq-sym-surjection}
    \left(E^{\otimes a}\right)^{\oplus b}\otimes(\det E)^{\otimes-c}\to \sym^{d}(\oo_X^{\oplus n}\otimes E^*)^,
\end{equation}
Therefore, the reduction of structure group $h$ is completely determined by the homomorphism
\begin{equation*}
    \sigma:E_{a,b,c}:=E(\Cabc)\cong \left(E^{\otimes a}\right)^{\oplus b}\otimes(\det E)^{\otimes-c}\xrightarrow{\phi_{a,b,c}}
    \sym^{d}(\oo_X^{\oplus n}\otimes E^*)^G\xrightarrow{h}\oo_X,
\end{equation*}
where $E(\Cabc)$ is the associated vector bundle via the representation $\rhoabc:\GL(n,\C)\to\GL(\Cabc)$ induced by the standard one. 
By (\ref{eq-iso-vs-det-1(0)}), $\sigma$ is not the zero section.
Moreover, since $(E,h,f)$ is a pseudo-equivariant $G$-bundle, (\ref{eq-pseudo-equivariant-f-h}) holds, which implies (\ref{eq-pseudo-equivariant-pairs-f-sigma}). Therefore, $(E,\sigma,f)$ is a pseudo-equivariant $\rhoabc$-pair. 

Summing up, the following holds.

\begin{proposition}\label{prop-pseudo-are-pairs}
There is an embedding of the category of pseudo-equivariant $G$-bundles into the category of pseudo-equivariant $\rhoabc$-pairs sending $(E,h,f)$ to $(E,\sigma,f)$, where
\begin{equation}\label{eq-def-sigma}
    \sigma:=h\circ\phi_{a,b,c}:E(\Cabc):=\iso(\oo_X^{\oplus n},E)\times_{\rhoabc}\Cabc\to \oo_X.
\end{equation}
\end{proposition}

\begin{remark}
    Due to Remark \ref{remark-degree-0}, one may think that it is enough to consider $\rho_{a,b}$-pairs, rather than $\rhoabc$-pairs, i.e. set $c=0$. However, the isomorphism $\det E\cong \oo_X$ is not canonical, which is relevant when constructing the moduli space.
\end{remark}

\subsection{Semistable pseudo-equivariant \texorpdfstring{$G$}{G}-bundles as semistable pseudo-equivariant pairs}\label{section-semistable-pseudo-equivariant-as-semistable-pairs}

This section is devoted to the proof of the following result, which compares the stability notions for pseudo-equivariant $G$-bundles and pseudo-equivariant $\rhoabc$-pairs.

\begin{proposition}\label{prop-semistable-pseudo-vs-pairs}
For $\chi\gg\delta\gg0$, the following holds: given a pseudo-equivariant $G$-bundle $(E,h,f)$, the corresponding pseudo-equivariant pair $(E,\sigma,f)$ is $(\delta,\chi)$-(semi)stable if and only if $(E,h,f)$ is (semi)stable.
\end{proposition}

We break the proof in several lemmas.

\begin{lemma}\label{lemma-mu-f-ge-0}
Fix a pseudo-equivariant $G$-bundle $(E,h,f)$ with associated pseudo-equivariant pair $(E,\sigma,f)$. Then $\mu(\eb,\alb,f)\ge 0$ for any weighted filtration, with $\mu(\eb,\alb,f)= 0$ if and only if the filtration $\eb$ is $f$-invariant.
\end{lemma}
\begin{proof}
Let $\lambda$ be a 1-parameter subgroup of $\GL(n,\C(X))$ determining $(\eb,\alb)$, where $\C(X)$ is the field of functions on $X$. It is clear that the weights appearing in the decomposition of $f\vert_\eta$ under the action of $\lambda$ are precisely the negative of the weights of $f^{*-1}\vert_\eta$. By Remark \ref{remark-properties-mu} and Definition (\ref{eq-def-mu}), $\mu(\eb,\alb,f)\ge 0$. Moreover, $\mu(\eb,\alb,f)= 0$ if and only if all the weights of $f\vert_\eta$ are $0$, which happens if and only if $\eb$ is $f$-invariant.
\end{proof}

Following \cite{schmitt}, we introduce a concept closely related to $(\delta,\chi)$-(semi)stability for pseudo-equivariant pairs.

\begin{definition}\label{def-asymptotic-(semi)stability}
A pseudo-equivariant pair $(E,\sigma,f)$ is called \textbf{asymptotically (semi)stable} if both $\mu(\eb,\alb,\sigma)$ and $\mu(\eb,\alb,f)$ are non-negative for every weighted filtration $(\eb,\alb)$ and, moreover, every proper weighted filtration such that $\mu(\eb,\alb,\sigma)=\mu(\eb,\alb,f)=0$ satisfies $M(\eb,\alb)\geqp 0$.
\end{definition}

\begin{lemma}\label{lemma-semistable-iff-asymptotically-semistable}
    A pseudo-equivariant $G$-bundle $(E,h,f)$ is (semi)stable if and only if the corresponding pseudo-equivariant pair $(E,\sigma,f)$ is asymptotically (semi)stable.
\end{lemma}
\begin{proof}
First assume that $(E,\sigma,f)$ is asymptotically (semi)stable. Fix a $\Gamma$-invariant maximal compact subgroup of $G$ with Lie algebra $\lie k\subset \gl$. For every $s\in i\lie k^{\Gamma}$, we may define parabolic subgroups $\widetilde{P}_s\subset\GL(n,\C)$ and $P_s:=\widetilde{P}_s\cap G$, where 
\begin{equation}
    \widetilde{P}_s:=\{g\in \GL(n,\C)\suhthat e^{ts}ge^{-ts}\,\text{remains bounded as}\;t\to\infty\}.
\end{equation}
Let $\bundle:=\iso(\oo_X^{\oplus n},E)$ be the bundle of frames of $E$ --- a $\GL(n,\C)$-bundle ---, and let $\bundle_G$ be the $G$-bundle associated to $h$. Let $\tau\in H^0(X,\bundle_G/P_s)^{\Gamma}$ be a $\Gamma$-invariant reduction of structure group with corresponding $P_s$-principal bundle $\bundle_\tau$. We have to prove that $\deg \bundle_G(\tau,s)\geqp0$. 

The extension of structure group of $\bundle_\tau$ to $\widetilde{P}_s$ provides a reduction of structure group of the bundle of frames of $E$ to $\widetilde{P}_s$, which we denote by $\widetilde\tau\in H^0(X,\bundle/\widetilde{P}_s)$, with corresponding $\widetilde{P}_s$-bundle $\bundle_{\widetilde\tau}$. This determines a filtration $\eb$ of $E$, and the 1-parameter subgroup
\begin{equation*}
    \C^*\to\GL(n,\C);\,t\mapsto e^{ts}
\end{equation*}
determines a corresponding system of weights $\alb$, thus defining a weighted filtration $(\eb,\alb)$ such that $M(\eb,\alb)=\deg \bundle_G(\tau,s)$. By assumption, $M(\eb,\alb)\geqp 0$ as long as $\mu(\eb,\alb,\sigma)=\mu(\eb,\alb,f)=0$. 

Let $\eta$ be the generic point of $X$. Then there is an isomorphism $\bundle_G\vert_\eta\cong G\times_\C\C(X)$ sending $\bundle_\tau\vert_\eta$ to $P_s\times_\C\C(X)$, where $\C(X)$ is the field of functions of $X$. Under this isomorphism $\bundle_\tau\vert_\eta$ is determined by the one parameter subgroup generated by the exponential of $s\otimes 1\in \lie p_s\otimes \C(X)$. This extends to an isomorphism $\bundle\vert_\eta\cong \GL(n,\C(X))$ such that $\bundle_{\widetilde\tau}\vert_\eta$ is also determined by $s\otimes 1$. Hence, we may choose the 1-parameter subgroup $\lambda$ of $\GL(E\vert_\eta)$ determining $(\eb,\alb)$ to be generated by the exponential of $s\otimes 1$. The $\C(X)$-action defined by $\lambda$ preserves $\bundle_G\vert_\eta$ --- in other words, it preserves $h\vert_\eta$ and $\sigma\vert_\eta$, hence $\mu(\sigma,\lambda)=0$.

Finally, the fact that $\tau$ is $\Gamma$-invariant implies that $\widetilde\tau$ is $\Gamma$-invariant and so is $\lambda$, hence we conclude that $\mu(f\vert_\eta,\lambda)=0$ by Lemma \ref{lemma-mu-f-ge-0}. Therefore, $\mu(\eb,\alb,\sigma)=\mu(\eb,\alb,f)=0$ and so $\deg \bundle_G(\tau,s)=M(\eb,\alb)\geqp0$ as required. 

To prove the other direction, let $(E,h,f)$ be a (semi)stable pseudo-equivariant $G$-bundle. Let $(\eb,\alb)$ be a proper weighted filtration of $E$, assume without loss that $\alb$ is integral and let $\lambda:\mathbb G_m(\C(X))\to \GL(n,\C(X))$ be a 1-parameter subgroup determining $(\eb,\alb)$. By Proposition \ref{prop-semistability-underlying-bundle}, the underlying $G$-bundle $(E,h)$ is semistable. By \cite[Corollary 2.4.4.6]{schmitt}, for $\delta\gg 0$ the underlying $\rhoabc$-swamp $(E,\oo_X,\sigma)$ is $\delta$-semistable --- in the sense of \cite[Section 2.3.2]{schmitt}, see Remark \ref{remark-pairs-vs-swamps} ---, hence $\mu(\eb,\alb,\sigma)\ge 0$ for $\delta\gg0$ by \cite[Proposition 2.3.6.5]{schmitt}.  Moreover $\mu(\eb,\alb,f)\ge 0$ by Lemma \ref{lemma-mu-f-ge-0}, so the first condition in Definition \ref{def-asymptotic-(semi)stability} holds.

Now assume that $\mu(\eb,\alb,\sigma)=\mu(\eb,\alb,f)=0$, i.e. $\mu(\sigma\vert_\eta,\lambda)=\mu(f\vert_\eta,\lambda)=0$. Then $h\vert_\eta$ is invariant under the $\C^*$-action given by $\lambda$.
Therefore, $\lambda$ determines a parabolic subgroup $P$ of $G$ and a reduction of structure group $\tau\in H^0(X,\bundle_G/P)$ such that the reduction of structure group of $\bundle=\iso(\oo_X^{\oplus n},E)$ determined by $\eb$ is an extension of structure group of $\bundle_\tau$. By \cite[Proposition 2.4.9.1]{schmitt} we may also choose $s\in\lie g$ such that $P=P_s$ and the 1-parameter subgroup of $\GL(n,\C(X))$ generated by the exponential of $s\otimes 1$ determines the weighted filtration $(\eb,r\alb)$ for some $r\in\Q_{\ge 0}$, so that $\deg \bundle_G(\tau,s)=rM(\eb,\alb)$. Note that $s\ne 0$, since $\eb$ is a proper filtration of $E$. Therefore, since $G$ is semisimple, $P_s\subset G$ is a proper parabolic subgroup.


It is left to show that $\deg \bundle_G(\tau,s)\geqp0$. Recall from Section \ref{section-semistable-twisted-equivariant-G-bundles} that a $\Gamma$-invariant compact subgroup $K$ has been previously fixed in order to define the stability notions for twisted equivariant $G$-bundles --- or pseudo-equivariant $G$-bundles. We may also choose a $\Gamma$-invariant maximal torus $T\subset G$ such that $T\cap K$ is a maximal torus of $K$. First note that, given $g\in G$, we may use the $G$-action on $\bundle$ to define a reduction of structure group $\tau g\in H^0(X,\bundle/g^{-1}P_sg)$ to $g^{-1}P_sg$, with corresponding weighted filtration $(\eb g,\alb)$. The induced pseudo-equivariant action $g^{-1}fg$ satisfies $\mu(\eb g,\alb,g^{-1}fg)=\mu(\eb,\alb,f)=0$. Therefore, we may assume after conjugating by an element $g\in G$ that $s\in i\lie k$ and the corresponding antidominant character $\chi_s$ is in a $\Gamma$-invariant Weyl chamber. 

Let $\ttau\in H^0(X,\bundle/\widetilde{P}_s)$ be the reduction of structure group given by extending the structure group of $\bundle_\tau$ to $\widetilde{P}_s$. By the assumption $\mu(\eb,\alb,f)=0$ and Lemma \ref{lemma-mu-f-ge-0}, we know that $\bundle_{\tau'}$ is invariant by the equivariant action given by $f$, i.e. $\fg$ sends $\bundle_{\tau'}$ to $\gamma^*\bundle_{\tau'}\subset\gamma^*\bundle$. By (\ref{eq-pseudo-equivariant-f-h}) $\fg(\bundle_\tau)\subset \gamma^*\bundle_\tau(1,\gamma)^{-1}$, which implies that the intersection of $\widetilde{P}_s$ with $G(1,\gamma)^{-1}$ inside of $\GL(n,\C)$ is non-empty for each $\gamma\in\Gamma$. 
Therefore, the subgroup $\overline{P}_s:=\widetilde{P}_s\cap (G\times_{(\theta,c)}\Gamma)\subset G\times_{(\theta,c)}\Gamma\subset\GL(n,\C)$, which has connected component of the identity $P_s$, intersects all the connected components of $G\times_{(\theta,c)}\Gamma$. 
Since $P_s$ is a normal subgroup of $\overline{P}_s$, there exists $g_\gamma\in G(1,\gamma)$ --- namely, any element $g_\gamma\in \overline{P}_s\cap G(1,\gamma)$ --- such that $g_{\gamma}P_sg_{\gamma}^{-1}=P_s$, for each $\gamma\in\Gamma$. Let $L_s$ be the levi subgroup of $P_s$ defined by (\ref{eq-def-levi}), and let
\begin{equation}
    R_s:=\{g\in G\suhthat \lim_{t\to\infty} e^{ts}ge^{-ts}=1\},
\end{equation}
be the radical of $P_s$ --- that is, its maximal solvable normal subgroup. Since every two Levi subgroups of $P_s$ are conjugate we may assume, after conjugating $g_{\gamma}$ by an element of $P_s$, that $\Int_{g_{\gamma}}$ preserves $L_s$. Since every two maximal tori in $L_s$ are conjugate to each other, after conjugating $g_{\gamma}$ further by an element of $L_s$ we may assume that $\Int_{g_{\gamma}}$ preserves the maximal torus $T\cap L_s$. Note that $\Int_{g_{\gamma}}$ preserves $R_s$, and in particular its maximal torus $T\cap R_s$. As $T\cap R_s$ and $T\cap L_s$ generate $T$, we may assume that $\Int_{g_{\gamma}}$ preserves $T$. Since $\Int_{g_{\gamma}}$ and $\Int_{(1,\gamma)}=\tg$ are in the same outer class, we conclude that $\Int_{g_{\gamma}}=\tg\Int_{g_{0,\gamma}}$ for some $g_{0,\gamma}\in G$ preserving $T$.

Therefore, $\tg\Int_{g_{0,\gamma}}$ preserves $P_s$. As $\tg$ preserves the Weyl chamber where $\chi_s$ lies, the element of the Weyl group defined by $\Int_{g_{0,\gamma}}$ is trivial. In particular, the $\tg\Int_{g_{0,\gamma}}$-invariance of $T\cap L_s$ and $T\cap R_s$ implies that both $T\cap L_s$ and $T\cap R_s$ are $\tg$-invariant, which in turn implies that $L_s$ and $P_s$ are $\tg$-invariant for each $\gamma\in\Gamma$. Therefore $s_{\Gamma}:=\sum_{\gamma\in\Gamma}\tg(s)$, which is non-zero because sums of non-zero vectors in a Weyl chamber are never zero, is in $i\lie k^{\Gamma}$ and satisfies $P_{s_{\Gamma}}=P_s$. By (semi)stability of $(E,h,f)$ we know that $\deg \bundle_G(\tau,s_\Gamma)\geqp0$. Moreover, (\ref{eq-pseudo-equivariant-f-h}) implies $F_A=\gamma^*\tg^{-1}(F_A)$ for each $\gamma\in\Gamma$, where $F_A$ is the curvature of a connection $A$ on $K\cap L_s$. Therefore,
\begin{align*}
    \deg \bundle_G(\tau,s_\Gamma)&=\frac i{2\pi}\int_X\langle s_{\Gamma},F_A\rangle
    \\&=\frac i{2\pi}\sum_{\gamma\in\Gamma}\int_X\langle \tg(s),F_A\rangle
    \\&=
    \frac i{2\pi}\sum_{\gamma\in\Gamma}\int_X
    \langle \tg(s),\gamma^{*-1}\tg F_A\rangle
    \\&=
    \frac i{2\pi}\sum_{\gamma\in\Gamma}\int_X
    \langle \tg(s),\tg F_A\rangle
    \\&=
    \frac i{2\pi}\sum_{\gamma\in\Gamma}\int_X
    \langle s, F_A\rangle
    \\&=
    \vert\Gamma\vert \deg \bundle_G(\tau,s),
\end{align*}
where the equation before the last one follows by $G$-invariance of the pairing $\langle \cdot, \cdot\rangle$.
In particular $\deg \bundle_G(\tau,s)\geqp0$, as required.
\end{proof} 

\begin{lemma}\label{lemma-asymptotically-semistable-implies-delta,chi-semistable}
For $\chi\gg\delta\gg0$, a pseudo-equivariant pair $(E,\sigma,f)$ coming from a pseudo-equivariant $G$-bundle $(E,h,f)$ is asymptotically (semi)stable only if it is $(\delta,\chi)$-(semi)stable.
\end{lemma}
\begin{proof}
Assume that $(E,\sigma,f)$ is asymptotically (semi)stable. By Lemma \ref{lemma-semistable-iff-asymptotically-semistable}, the corresponding pseudo-equivariant $G$-bundle is (semi)stable, and so is the underlying $G$-bundle by Proposition \ref{prop-semistability-underlying-bundle}. By \cite[Corollary 2.4.4.6]{schmitt}, for $\delta\gg 0$ the corresponding $\rhoabc$-swamp $(E,\oo_X,\sigma)$ is $\delta$-semistable. By \cite[Theorem 2.3.4.1]{schmitt} and Remark \ref{remark-degree-0}, there exists a constant $C$, independent of $(E,\sigma,f)$, such that 
\begin{equation}\label{eq-bounded}
\deg(F)/\rk(F)\le C .   
\end{equation}

By Lemma \ref{lemma-delta,chi-semistable-finite-set}, $(\delta,\chi)$-(semi)stability only has to be verified on weighted filtrations $(\eb,\alb)$ such that $(\rk\eb,\alb)\in S$, for some finite set $S$. We may assume that $S$ only features integral weights. Using (\ref{eq-bounded}), we may find a negative number $-\epsilon_0$ such that $M(\eb,\alb)>-\epsilon_0$ for such weighted filtrations. Choose $\delta\ge \epsilon_0$. Using \cite[Lemma 1.5.1.41]{schmitt}, we may find a negative number $-\epsilon_1$ such that $M(\eb,\alb)+\delta\mu(\eb,\alb,\sigma)> -\epsilon_1$, for every weighted filtration $(\eb,\alb)$ such that $(\rk\eb,\alb)\in S$. Choose $\chi\ge\epsilon_1$. Then, for such weighted filtrations, since the weights are integral we find
\begin{equation*}
    M(\eb,\alb)+\delta\mu(\eb,\alb,\sigma)+\chi\mu(\eb,\alb,f)>0
\end{equation*}
whenever $\mu(\eb,\alb,f)$ or $\mu(\eb,\alb,\sigma)>0$. If, on the other hand, $\mu(\eb,\alb,f)=\mu(\eb,\alb,\sigma)=0$ --- this is the only remaining case, since $\mu(\eb,\alb,f),\mu(\eb,\alb,\sigma)\ge0$ by asymptotic (semi)stability ---,
then 
\begin{equation*}
    M(\eb,\alb)+\delta\mu(\eb,\alb,\sigma)+\chi\mu(\eb,\alb,f)=M(\eb,\alb)\geqp0
\end{equation*}
by Definition \ref{def-asymptotic-(semi)stability}, as required.
\end{proof}

\begin{lemma}\label{lemma-semistable-if-delta,chi-semistable}
    A pseudo-equivariant $G$-bundle $(E,h,f)$ is (semi)stable if the corresponding pseudo-equivariant pair $(E,\sigma,f)$ is $(\delta,\chi)$-(semi)stable.
\end{lemma}
\begin{proof}
    Note that, given a weighted filtration $(\eb,\alb)$ such that 
    \begin{equation}\label{eq-mu-sigma-mu-f-=-0}
        \mu(\eb,\alb,\sigma)=\mu(\eb,\alb,f)=0,
    \end{equation}
    then
    \begin{equation*}
        M(\eb,\alb)+\delta\mu(\eb,\alb,\sigma)+\chi\mu(\eb,\alb,f)=M(\eb,\alb)
    \end{equation*} 
    for every $\delta$ and $\chi$. Therefore, if  $(E,\sigma,f)$ is $(\delta,\chi)$-(semi)stable, then 
    (\ref{eq-mu-sigma-mu-f-=-0}) implies that $M(\eb,\alb)\geqp0$. Using this fact, the proof is identical to the first part of the proof of Lemma \ref{lemma-semistable-iff-asymptotically-semistable}.
\end{proof}

Proposition \ref{prop-semistable-pseudo-vs-pairs} follows from Lemmas \ref{lemma-semistable-iff-asymptotically-semistable}, \ref{lemma-asymptotically-semistable-implies-delta,chi-semistable} and \ref{lemma-semistable-if-delta,chi-semistable}.

\section{The parameter space of twisted equivariant \texorpdfstring{$G$}{G}-bundles}\label{section-parameter-space}

Fix a compact Riemann surface $X$, a finite group $\Gamma$ of complex automorphisms of $X$, a connected semisimple complex Lie group $G$ with center $Z$, a homomorphism $\theta:\Gamma\to\Aut(G)$ and a 2-cocycle $c\in Z^2_{\theta}(\Gamma,Z)$. By Propositions \ref{prop-equivalence-of-categories-twisted-pseudo}, \ref{prop-pseudo-are-pairs} and \ref{prop-semistable-pseudo-vs-pairs}, if we construct the moduli space of $(\delta,\chi)$-semistable pseudo-equivariant $\rhoabc$-pairs then the moduli space of semistable $(\theta,c)$-twisted $\Gamma$-equivariant $G$-bundles will be embedded in it for $\chi\gg\delta\gg0$.

Throughout this section, given two spaces $Y$ and $Z$, we denote by $\pi_Y:Y\times Z\to Y$ the projection.

\subsection{Bounded families of vector bundles and the Quot Scheme}\label{section-bounded}

Recall that a family of vector bundles $\mathfrak G$ is \textbf{bounded} if there exists a scheme $S$ of finite type over $\C$ and a vector bundle $E_S$ over $S\times_{\C} X$ such that, for every vector bundle $E$ in $\mathfrak G$, there exists $s\in S$ such that $E_S\vert_{\{s\}\times X}\cong E$. 

Assume that $\chi\gg\delta\gg0$, so that we may use Proposition \ref{prop-semistable-pseudo-vs-pairs}.
By \cite[Proposition 4.1]{oscar-suratno} the family of semistable twisted equivariant $G$-bundles over $X$ has an underlying family of semistable $G$-bundles, whose associated family $\mathfrak G$ of vector bundles of rank $n$ and trivial determinant --- obtained via extension of structure group to $\GL(n,\C)$ --- is bounded by \cite[Corollary 2.4.4.6, Proposition 2.2.3.7 and Theorem 2.3.4.1]{schmitt}. Let $\oo_X(1)$ be a $\Gamma$-invariant ample line bundle on $X$; this can be achieved by taking any ample line bundle and tensoring it by its pullbacks with respect to the elements of $\Gamma$. Set $E(m):=E\otimes\oo_X(1)^{\otimes m}$, where $E$ is a vector space or a vector bundle over $X$. By \cite[Corollary 2.2.3.4]{schmitt} and Remark \ref{remark-degree-0}, we know that there exist a complex vector space $V$ and a natural number $m$ satisfying the following.

\begin{proposition}\label{prop-twisted-equivariant-are-quotients}
Every vector bundle $E$ of rank $n$ and trivial determinant associated to a semistable twisted equivariant $G$-bundle is a quotient $V(-m)\to E$ over $X$ such that the induced homomorphism $V\to H^0(E(m))$ on spaces of global sections is an isomorphism.
\end{proposition}

There is a Grothendieck's \textbf{Quot Scheme} $\qq$ equipped with a family of vector bundles $\f$ over $\qq\times X$ of rank $n$ and trivial determinant over $X$, and a quotient homomorphism $q:\pi_X^*V(-m)\to \f$ of sheaves over $\qq\times X$. The quasiprojective variety $\qq$ is a fine moduli space for equivalence classes of vector bundles of rank $n$ and trivial determinant which are quotients of $V(-m)$ over $X$, where the equivalence relation is given by having the same kernel. By Proposition \ref{prop-twisted-equivariant-are-quotients}, for every $(\delta,\chi)$-semistable $\rhoabc$-pair $(E,\sigma,f)$ coming from a twisted equivariant $G$-bundle, there exists $t\in\qq$ such that $\f\vert_{\{t\}\times X}\cong E$.

\subsection{The parameter space of \texorpdfstring{$G$}{G}-bundles}\label{section-parameter-space-pairs}

In \cite{schmitt} the moduli space of holomorphic principal $G$-bundles is constructed using \textbf{$\rho$-swamps}, which are triples $(E,L,\phi)$ made of a vector bundle $E$, a line bundle $L$ and a surjection $\phi:E(W)\to L$, where $\rho: \GL(n,\C)\to W$ is some representation and $E(W)=\hhom(\oo_X^{\oplus n},E)\times_{\rho}W$ is the associated vector bundle. Indeed, if we ignore the pseudo-equivariant structure in Section \ref{section-pseudo-equivariant-bundles-as-pairs}, we obtain an embedding of the category of $G$-bundles into the category of \textbf{$\rhoabc$-pairs} $(E,\sigma)$ consisting of a vector bundle $E$ of rank $n$ and degree $0$, and a homomorphism $\sigma: E(\Cabc)\to\oo_X$ --- these are just $\Gamma$-pseudo-equivariant $\rhoabc$-pairs for $\Gamma=1$. In what follows we build a parameter space for this pairs, following \cite[Section 2.3]{schmitt}.

Given a vector bundle $E$, we write $E_{a,b,c}:=(E^{\otimes a})^{\oplus b}\otimes(\det E)^{\otimes -c}$. Let $k$ be big enough so that $V_k:=\pi_{\qq*}(\pi_X^*V(-m)_{a,b,c}\otimes \pi_X^*\oo_X(ak))=\pi_{\qq*}(\pi_X^*V_{a,b,c}\otimes \pi_X^*\oo_X(a(k-m)))$ and $\pi_{\qq*}\pi_X^*\oo_X(ak)$ are vector bundles, where $\pi_{\qq}:\qq\times X\to\qq$ is the projection. 
Let $\mathfrak B:=\PP\left(\hhom(V_k,\pi_{\qq*}\pi_X^*\oo_X(ak))\right)$, which comes equipped with a projective bundle projection $p_{\mathfrak B}:\mathfrak B\to\qq$. There is a tautological homomorphism 
$$(p_{\mathfrak B}\times \id_X)^*\pi_X^*V(-m)_{a,b,c}\otimes (p_{\mathfrak B}\times\id_X)^*\pi_X^*\oo_X(ak)\to (p_{\mathfrak B}\times\id_X)^*\pi_X^*\oo_X(ak)\otimes\oo_{\mathfrak B}(1)$$
or, in other words, a homomorphism
$$\phi':(p_{\mathfrak B}\times \id_X)^*\pi_X^*V(-m)_{a,b,c}\to \pi_{\mathfrak B}^*\oo_{\mathfrak B}(1).$$

Let $\f_{a,b,c}:=(\f^{\otimes a})^{\oplus b}\otimes(\det\f)^{\otimes -c}$ and consider the map
\begin{equation*}
    q_{a,b,c}:(p_{\mathfrak B}\times\id_X)^*\pi_X^*V(-m)_{a,b,c}\to (p_{\mathfrak B}\times\id_X)^*\f_{a,b,c}
\end{equation*}
induced by $q$. 

\begin{remark}
    Note that every $\rhoabc$-pair providing a semistable $G$-bundle is isomorphic to 
$$((p_{\mathfrak B}\times \id_X)^*\f_{a,b,c}\vert_{\{t\}\times X},\phi\vert_{\{t\}\times X})$$
for some $t\in\mathfrak B$. Indeed, for every vector bundle $E$ of rank $n$ and trivial determinant over $X$, given a quotient $V(-m)\to E$ inducing an isomorphism on global sections $V\xrightarrow{\sim}H^0(E(m))$, any homomorphism $\sigma:E_{a,b,c}\to\oo_X$ lifts to a homomorphism $V(-m)_{a,b,c}\to\oo_X$ via the composition
\begin{equation*}
    V_{a,b,c}\xrightarrow{\sim}H^0(V_{a,b,c}\otimes\oo_X)\xrightarrow{\sim}H^0(E(m)_{a,b,c})\xrightarrow{H^0(\sigma)}H^0(\oo_X(m)^{\otimes a})\xrightarrow{\text{ev}}\oo_X(ma),
\end{equation*}
which induces a homomorphism
\begin{equation*}
    V(-m)_{a,b,c}\cong V_{a,b,c}(-ma)\to\oo_X.
\end{equation*}
We may thus apply Proposition \ref{prop-twisted-equivariant-are-quotients}.
\end{remark}

By \cite[Proposition 2.3.5.1]{schmitt}, the set of points $t\in\mathfrak B$ such that the restriction of $\phi'\vert_{\{t\}\times X}$ to the kernel of $q_{a,b,c}$ is zero is a closed subscheme ${\ii}_1\subset\mathfrak B$. On this subscheme, $\phi'$ induces a homomorphism 
$$\phi:(p_{\mathfrak B}\times \id_X)^*\f_{a,b,c}\vert_{{\ii}_1}\to (p_{\mathfrak B}\times\id_X)^*\pi_X^*\oo_X\vert_{{\ii}_1}.$$

Again by Proposition 2.3.5.1, we find a closed subscheme ${\ii}_0$ of ${\ii}_1$ where the restriction of $\phi$ to the kernel of 
$$(p_{\mathfrak B}\times \id_X)^*\f_{a,b,c}\to \sym^*(\oo_X^{\oplus n}\otimes (p_{\mathfrak B}\times \id_X)^*\f^*)^G$$
is zero. 

Finally, there is an open subvariety ${\ii}$ of ${\ii}_0$ consisting of the points $t\in{\ii}_0$ such that $\phi\vert_{\{t\}\times X}$ induces a $G$-bundle, namely those inducing a section of $\hhom(\oo_X^{\oplus n},(p_{\mathfrak B}\times \id_X)^*\f)\sslash G$ which actually lies in $\iso(\oo_X^{\oplus n},(p_{\mathfrak B}\times \id_X)^*\f)/G$.

We rename $\phi$ to be its restriction to $\ii\times X$. Set $V_{\ii}:=V\otimes\oo_{\ii\times X}$, $\f_{\ii}:=(p_{\mathfrak B}\times \id_X)^*\f\vert_{{\ii}\times X}$ and $q_{\ii}:=(p_{\mathfrak B}\times \id_X)^*q\vert_{{\ii}\times X}$. We have obtained a universal family $(\f_{\ii},\phi)$ over ${\ii}\times X$ parametrizing $G$-bundles over $X$ whose extension of structure group to $\GL(n,\C)$ has associated vector bundle isomorphic to a quotient of $V(-m)$. Note that, for each $\gamma\in\Gamma$, there are alternative universal families $((\id_{\ii}\times\gamma)^*\f_{\ii},(\id_{\ii}\times\gamma)^*\phi)$ and a quotient 
\begin{equation*}
    (\id_{\ii}\times\gamma)^*q_{\ii}:(\id_{\ii}\times\gamma)^*V_{\ii}(-m)\to(\id_{\ii}\times\gamma)^*\f_{\ii}.
\end{equation*}

\begin{remark}
    Note that, given two $G$-bundles $(E,h)$ and $(E,h')$ represented by the same point $t\in\mathfrak B$, the corresponding homomorphisms
    $$\left(V^{\otimes a}\right)^{\oplus b}\otimes(\det V)^{\otimes -c}\to \oo_X$$
    --- given by Section \ref{section-pseudo-equivariant-bundles-as-pairs} and $q$ ---
    are related by multiplication by an element of $\C^*$.
    In particular, they are in the same $\GL(V)$-orbit. We will see in Proposition \ref{prop-iso-vs-orbit} that this implies that $(E,h)\cong(E,h')$. In other words, working with $\mathfrak B=\PP\left(\hhom(V_k,\pi_{\qq*}\pi_X^*\oo_X(ak))\right)$ instead of $\hhom(V_k,\pi_{\qq*}\pi_X^*\oo_X(ak))$ does not lose information.
\end{remark}

\subsection{The parameter space of twisted equivariant pairs}\label{section-parameter-space-pseudo-equivariant-pairs}

With notation as in Section \ref{section-parameter-space-pairs}, let
$$\Vg\ii:=\bigoplus_{\gamma\in\Gamma}\hhom(\pi_{\ii*}V_{\ii},\pi_{\ii*}(\id_{\ii}\times\gamma)^*V_{\ii})
$$ 
and
\begin{equation*}
    \Vg\ii^\vee:=
    \bigoplus_{\gamma\in\Gamma}\hhom(\pi_{\ii*}V^*_{\ii},\pi_{\ii*}(\id_{\ii}\times\gamma)^*V^*_{\ii})
\end{equation*}
Set $\cc:=\PP(\Vg\ii\oplus \Vg\ii^\vee)$, and let $p_{\cc}:\cc\to \ii$ be the bundle projection. We endow $\cc$ with the $\GL(V)$-left action induced by
\begin{equation*}
    \GL(V)\times(\Vg\ii\oplus \Vg\ii^\vee)\to \Vg\ii\oplus \Vg\ii^\vee;\,g\cdot(\bigoplus_{\gamma\in\Gamma}\psi'_{\gamma},\bigoplus_{\gamma\in\Gamma}\psi''_{\gamma}):=(\bigoplus_{\gamma\in\Gamma}g\psi'_{\gamma}g^{-1},\bigoplus_{\gamma\in\Gamma}g^{t-1}\psi''_{\gamma}g^t).
\end{equation*}
This comes equipped with a tautological section $\psi'\oplus\psi''\in H^0((p_{\cc}\times\id_X)^*(\Vg\ii\oplus \Vg\ii^\vee)\otimes \oo_{\cc}(1))$. Let $\psi'\in H^0((p_{\cc}\times\id_X)^*\Vg\ii\otimes \oo_{\cc}(1))$ be its projection onto the first summand, and write
\begin{equation*}
    \psi'=\bigoplus_{\gamma\in\Gamma}\psi'_{\gamma}:(p_{\cc}\times\id_X)^*V_{\ii}(-m)\to\bigoplus_{\gamma\in\Gamma}(p_{\cc}\times\id_X)^*(\id_{\ii}\times\gamma)^*V_{\ii}(-m)\otimes\oo_{\cc}(1).
\end{equation*}
Note that the operations of pulling back by $\id_{\ii}\times\gamma$ and tensoring by $\oo_X(-m)$ commute by $\Gamma$-invariance of $\oo_X(1)$, so we don't need to especify the order.

Let $\cc_1\subset\cc$ be the subset of points $t$ such that the restriction of $\psi'_{\gamma}$ to $\{t\}\times X$ is an isomorphism for every $\gamma\in\Gamma$. Since the degrees of a line bundle over $X$ and its pullback by an element of $\Gamma$ are equal, this is the complement of the locus of points $t\in\cc$ such that the induced homomorphisms 
\begin{equation*}
     \det\psi'_{\gamma}:\det(p_{\cc}\times\id_X)^*V_{\ii}(-m)\to \det(p_{\cc}\times\id_X)^*(\id_{\ii}\times\gamma)^*V_{\ii}(-m)\otimes\oo_{\cc}(1)
\end{equation*}
are zero on $\{t\}\times X$, for each $\gamma\in\Gamma$. By \cite[Proposition 2.3.5.1]{schmitt} this locus is closed, so $\cc_1$ is open in $\cc$. 

Let $\psi_0$ be the image of $\psi'$ by the homomorphism
$$(p_{\cc}\times\id_X)^*\Vg\ii\to (p_{\cc}\times\id_X)^*\bigoplus_{\gamma\in\Gamma}\hhom(\ker q_{\ii},(\id_{\ii}\times\gamma)^*\f_{\ii})$$
of vector bundles over ${\ii}\times X$ given by (the pullback of) restricting each homomorphism $V_{\ii}(-m)\to(\id_{\ii}\times\gamma)^*V_{\ii}(-m)$ to $\ker q_{\ii}$ and composing with $(\id_{\ii}\times\gamma)^*q_{\ii}$. The locus of $\mathfrak C$ where $\psi_0$ is zero is a closed subscheme $\mathfrak C_2$ of $\mathfrak C$ by \cite[Proposition 2.3.5.1]{schmitt}.
This is precisely the subvariety where $\psi'$ descends to a section $\psi$ of 
\begin{equation*}
    \f_{\Gamma,\cc}:=(p_{\cc}\times\id_X)^*\bigoplus_{\gamma\in\Gamma}\hhom(\f_{\ii},(\id_{\ii}\times\gamma)^*\f_{\ii}).
\end{equation*}
Again by \cite[Proposition 2.3.5.1]{schmitt}, the locus of $\mathfrak C$ where the image of $\psi'$ by the homomorphism
\begin{align*}
    p_{\cc}^*\Vg\ii\to p_{\cc}^*\bigoplus_{(\gamma,\gamma')\in\Gamma\times\Gamma}\hhom(V_{\ii},(\id_{\ii}\times\gamma\gamma')^*V_{\ii});\\
    p_{\cc}^*(f_{\gamma})_{\gamma\in\Gamma}\mapsto\bigoplus_{\gamma\in\Gamma} p_{\cc}^*((\gamma^*f_{\gamma'})f_{\gamma}-f_{\gamma\gamma'})
\end{align*}
is $0$ is a closed subscheme $\mathfrak C_3\subset\cc$. We also consider the closed subscheme $\mathfrak C_4$ where $\psi'_1=\id$. 

Set $\f_{a,b,c,\jj}:=(p_{\mathfrak B}\times\id_X)^*\f_{a,b,c}\vert_{\jj}$. Let $\cc_5\subset\cc_2$ be the closed subscheme where the image of $p_{\cc}^*\phi$ by the homomorphism
\begin{align*}
    \hhom((p_{\cc}\times\id_X)^*\f_{a,b,c,\jj},\oo_{\cc\times X})\to 
    \bigoplus_{\gamma\in\Gamma}\hhom((p_{\cc}\times\gamma)^*\f_{a,b,c,\jj},\oo_{\cc\times X});
    \\
    \Phi
    \mapsto 
    \bigoplus_{\gamma\in\Gamma}(\psi'_{\gamma}(\Phi)-(\id_{\cc}\times\gamma)^*\Phi\cdot(1,\gamma)^{-1})   
\end{align*}
is $0$ --- note that $\psi'_{\gamma}(\Phi)\in \hhom((p_{\cc}\times\gamma)^*\f_{a,b,c,\jj},\oo_{\cc\times X})$ because we are working on $\cc_2$.

Finally, let $\cc_6\subset\cc_1$ be the subset of points $t$ such that the image of $\psi'\oplus\psi''$ by the homomorphism
\begin{equation*}
    (p_{\cc}\times\id_X)^*(\Vg\ii\oplus \Vg\ii^\vee)\otimes \oo_{\cc}(1)\to \Vg\ii\otimes \oo_{\cc}(1);\,
    \psi'\oplus\psi''\mapsto \psi''-\psi'^{*-1},
\end{equation*}
where 
\begin{equation}\label{eq-*-1}
\psi'^{*-1}=\bigoplus_{\gamma\in\Gamma}\psi'^{*-1}_{\gamma},    
\end{equation}
is $0$ --- note that we are using the notation $\psi'=\bigoplus_{\gamma\in\Gamma}\psi'_{\gamma}$. By \cite[Proposition 2.3.5.1]{schmitt}, this is a closed subscheme of $\cc_1$, so it is locally closed in $\cc$.

\begin{remark}
    To obtain the parameter space of twisted equivariant $G$-bundle, we restrict below to a closed subvariety $\cc_0\subset\cc_6$. From this point of view we can now see one of the reasons why we work with $\PP(\Vg\ii\oplus V^\vee_{\Gamma})$, rather than $\PP(\Vg\ii)$. Indeed, a point in $\PP(\Vg\ii)$ only determines a set of homomorphisms $V\otimes\oo_X\to\gamma^*V\otimes\oo_X$ up to rescaling by an element of $\C^*$ --- in other words, up to a choice of local trivialization of $\oo_{\cc}(1)$. However, given a point in $\cc_6\subset \PP(\Vg\ii\oplus V^\vee_{\Gamma})$, these homomorphisms are uniquely defined by the condition that $\psi''$ is the ``dual inverse'' of $\psi'$.
\end{remark}

Note that every vector bundle $E$ of rank $n$ and trivial determinant with a reduction of structure group to $G$ making it a semistable $G$-bundle, equipped with a set of homomorphisms $E\to\gamma^*E$ for $\gamma\in\Gamma$, is represented by a point in $\cc$, since every such homomorphism lifts to a homomorphism $V\to\gamma^*V$ via the isomorphism $V\cong H^0(E(m))$ between the spaces of global sections. More precisely, any such object is represented by a point in $\cc_2$. This implies, by Proposition \ref{prop-semistability-underlying-bundle}, that every semistable pseudo-equivariant $G$-bundle is represented by a point in $\cc_2$. Conversely, every point in $\cc_2$ represents a --- not necessarily semistable --- vector bundle $E$ of rank $n$ and trivial determinant with a reduction of structure group to $G$, equipped with a set of homomorphisms $E\to\gamma^*E$ for $\gamma\in\Gamma$. Inside $\cc_1\cap\mathfrak C_2$, all the homomorphisms $E\to\gamma^*E$ are isomorphisms. Moreover, points in $\cc_1\cap\mathfrak C_2\cap\cc_3\cap\cc_4$ and $\cc_1\cap\mathfrak C_2\cap\cc_5$ satisfy (\ref{eq-pseudo-equivariant-composition-f}) and (\ref{eq-pseudo-equivariant-f-h}), respectively. Therefore, every point in $\cc_1\cap\mathfrak C_2\cap\cc_3\cap\cc_4\cap\cc_5$ represents a pseudo-equivariant $G$-bundle, and any semistable pseudo-equivariant $G$-bundle is represented by a point in $\cc_1\cap\mathfrak C_2\cap\cc_3\cap\cc_4\cap\cc_5$. Via Proposition \ref{prop-equivalence-of-categories-twisted-pseudo}, the same statement is true after replacing pseudo-equivariant $G$-bundles with twisted equivariant $G$-bundles. Moreover, if we restrict to $\cc_1\cap\mathfrak C_2\cap\cc_3\cap\cc_4\cap\cc_5\cap\cc_6$, then all the semistable twisted equivariant $G$-bundles are still represented, and each set of isomorphisms $E
\to\gamma^*E
$ only appears once for each $G$-bundle. Set $\cc_0:=\cc_1\cap\cc_2\cap\cc_3\cap\cc_4\cap\cc_5\cap\cc_6$. 

Let $\f_{\cc}:=(p_{\cc}\times\id_X)^*\f_{\ii}$, $\f_{\cc_0}:=(p_{\cc}\times\id_X)^*\f_{\ii}\vert_{\cc_0\times X}$ and $\phi_{\cc_0}:=(p_{\cc}\times\id_X)^*\phi\vert_{\cc_0\times X}\in \Hom((\f_{\cc_0})_{a,b,c},\oo_{\cc\times X})$. Denote by
\begin{equation*}
    \psi_{\cc_0}\in H^0\left(\bigoplus_{\gamma\in\Gamma}\hhom(\f_{\cc_0},(\id_{\cc_0}\times\gamma)^*\f_{\cc_0})\right)
\end{equation*}
the section induced by $\psi'$. Then $(\f_{\cc_0},\phi_{\cc_0},\psi_{\cc_0})$ is a universal family of pseudo-equivariant pairs over $\cc_0\times X$. Let $V_{\cc}:=V\otimes\oo_{\cc\times X}$, which is equipped with a surjective homomorphism $q_{\cc}:=(p_{\cc}\times\id_X)^*q_{\ii}:V_{\cc}(-m)\to \f_{\cc}$.

The next proposition reduces the problem of constructing the moduli space of twisted equivariant $G$-bundles to the study of a GIT quotient $\cc_0\sslash\GL(V)$.

\begin{proposition}\label{prop-iso-vs-orbit}
The natural action of $\GL(V)$ on $V$ induces an action on $\cc$ preserving $\cc_0\subset \cc$. For every $t$ and $s\in\cc_0$, the  pseudo-equivariant $\rhoabc$-pairs $(\f_{\cc_0},\phi_{\cc_0},\psi_{\cc_0})\vert_{\{t\}\times X}$ and $(\f_{\cc_0},\phi_{\cc_0},\psi_{\cc_0})\vert_{\{s\}\times X}$ are isomorphic if and only if they are in the same $\GL(V)$-orbit.
\end{proposition}
\begin{remark}
    By Propositions \ref{prop-equivalence-of-categories-twisted-pseudo} and \ref{prop-pseudo-are-pairs}, two twisted equivariant $G$-bundles are isomorphic if and only if the corresponding pseudo-equivariant pairs are, so Proposition \ref{prop-iso-vs-orbit} implies that two twisted equivariant $G$-bundles are isomorphic if and only if the corresponding points of $\cc_0$ are in the same $\GL(V)$-orbit.
\end{remark}
\begin{proof}
The $\GL(V)$-invariance of $\cc_0$ follows from its definition.
Given $t\in\cc_0$, consider the quotient $q_{\cc}\vert_{\{t\}\times X}:V_{\cc}(-m)\vert_{\{t\}\times X}\to\f_{\cc}\vert_{\{t\}\times X}$. An element $g\in\GL(V)$ sends $\ker q_{\cc}\vert_{\{t\}\times X}\subset V(-m)$, which is a subbundle of rank $\dim V-n$, to another subbundle of $V(-m)$, and the quotient is equal to $q_{\cc}\vert_{\{s\}\times X}:V_{\cc}(-m)\vert_{\{s\}\times X}\to\f_{\cc}\vert_{\{s\}\times X}$ for some $s\in\cc$. Moreover, the automorphism of $V_{\cc}$ determined by $g$ induces an isomorphism of vector bundles $\f_{\cc_0}\vert_{\{t\}\times X}\xrightarrow{\sim}\f_{\cc_0}\vert_{\{s\}\times X}$. Via this isomorphism, the pseudo-equivariant pair $(\f_{\cc_0},\phi_{\cc_0},\psi_{\cc_0})\vert_{\{t\}\times X}$ determines another one $(\f_{\cc_0},\phi_{\cc_0},\psi_{\cc_0})\vert_{\{s\}\times X}$, which determines $s$ uniquely.

It is left to prove that, given $t$ and $s\in\cc$ such that the corresponding pseudo-equivariant pairs $(\f_{\cc_0},\phi_{\cc_0},\psi_{\cc_0})\vert_{\{t\}\times X}$ and 
$(\f_{\cc_0},\phi_{\cc_0},\psi_{\cc_0})\vert_{\{s\}\times X}$ are isomorphic, they are in the same $\GL(V)$-orbit. An isomorphism induces an isomorphism of vector spaces $H^0(\f_{\cc}(m)\vert_{\{t\}\times X})\cong H^0(\f_{\cc}(m)\vert_{\{s\}\times X})$ which in turn, via the isomorphism of vector spaces of sections $H^0(q_{\cc}\otimes\id_{\oo_X(m)})$, induces an automorphism of $V$, i.e. an element of $\GL(V)$. It is clear that this induces the isomorphism that we started with, as required.
\end{proof}

\section{The moduli space for \texorpdfstring{$G$}{G} semisimple}\label{section-moduli-space-semisimple}

We maintain the notation of the previous section. In order to perform a GIT quotient $\cc_0\sslash\GL(V)$ providing a moduli space parametrizing isomorphism classes of polystable $(\theta,c)$-twisted equivariant $G$-bundles, we need to find a linearization of the $\GL(V)$-action on some very ample line bundle over $\cc_0$, such that the corresponding GIT definition of (semi, poly)stablity is equivalent to the definitions of Section \ref{section-semistable-twisted-equivariant-G-bundles}. This section is devoted to showing that such a linearization exists, and that the resulting moduli space is a complex projective variety.

\subsection{Gieseker space and linearization}\label{section-gieseker}

In order to find an appropriate linearization of the $\GL(V)$-action on $\cc_0$ and to apply the Hilbert--Mumford criterion, we use an injective morphism into a projective space, usually called a Gieseker space --- the role of Gieseker spaces in the construction of the moduli space of $G$-bundles is explained in \cite[Section 2.3.5]{schmitt}.

Let $\Gg_1:=\Hom(\bigwedge^n V\otimes\oo_X,\oo_X(nm))$ and $\GG_1:=\PP(\Gg_1)$ --- a complex projective space, 
equipped with the very ample line bundle $\oo_{\GG_1}(1)$.
With notation as in Section \ref{section-parameter-space-pseudo-equivariant-pairs},
let $\Ll_{\cc}\to \cc$ be the line bundle that satisfies $\bigwedge^n\f_{\cc}\cong\pi_{\cc}^*\Ll_{\cc}$ --- this exists because the determinant of $\f_{\cc}\vert_{\{s\}\times X}$ is trivial for every $s\in\cc$, see Section \ref{section-bounded}. Then 
the surjective homomorphism $q_\cc:V_{\cc}(-m)\to\f_{\cc}$ induces a homomorphism 
\begin{equation*}
    \bigwedge^n (q_\cc\otimes\id_{\pi_X^*\oo_X(nm)}):\bigwedge^n V\otimes\oo_{\cc\times X}\to \pi_X^*\oo_X(nm)\otimes\pi_{\cc}^*\Ll_{\cc}.
\end{equation*}
This, in turn, determines a morphism 
\begin{equation*}
    \Gies_1:\cc\to\GG_1
\end{equation*}
such that $\Gies_1^*\oo_{\GG_1}(1)\cong\Ll_{\cc}$.

Now let $V_{a,b}:=\left(V^{\otimes a}\right)^{\oplus b}=V_{a,b,0}$,
\begin{equation*}
    \Gg_2:=\Hom\left(V_{a,b}\otimes \oo_X, \oo_X(ma)\right)
\end{equation*}
and $\GG_2:=\PP(\Gg_2)$ --- a complex projective space, equipped with the very ample line bundle $\oo_{\GG_2}(1)$. Then the homomorphism $\phi_{\cc}:(\f_{\cc})_{a,b,c}\to \pi_X^*\oo_X$ induces a homomorphism
\begin{equation*}
    V_{a,b}\otimes\oo_{\cc\times X}\to(\f_{\cc})_{a,b}\otimes\pi_X^*(\oo_X(ma))\to \pi_X^*(\oo_X(ma))\otimes\pi_{\cc}^*(\Ll^{\otimes c}_{\cc}),
\end{equation*}
which in turn induces a morphism $\Gies_2:\cc\to\GG_2$ such that $\Gies_2^*\oo_{\GG_2}(1)\cong \Ll_{\cc}^{\otimes c}$. 

Finally, let
\begin{equation*}
    V_{\Gamma}:=\bigoplus_{\gamma\in\Gamma}\Hom(V\otimes\oo_X,\gamma^*(V\otimes\oo_X))\andd
    V_{\Gamma}^\vee:=\bigoplus_{\gamma\in\Gamma}\Hom(V^*\otimes\oo_X,\gamma^*(V^*\otimes\oo_X)),
\end{equation*}
and set
$\Gg_3:=V_{\Gamma}\oplus V^\vee_{\Gamma}$ and $\GG_3:=\PP(\Gg_3)$. The tautological section $\psi'\oplus\psi''\in (p_{\cc}\times\id_X)^*(\Vg\ii\oplus \Vg\ii^\vee)$ determines a morphism $\Gies_3:\cc\to\GG_3$. 

The product
\begin{equation*}
    \Gies:=\Gies_1\times\Gies_2\times\Gies_3:\cc\to\GG:=\GG_1\times\GG_2\times\GG_3
\end{equation*} 
is an injective $\GL(V)$-equivariant morphism. Of course, it restricts to an injective morphism on the locally closed subvariety $\cc_0$. We refer to $\GG$ as the \textbf{Gieseker space}, and to $\Gies$ as the \textbf{Gieseker morphism}.

Given positive integers $s_1,s_2$ and $\chi$, let $\oo_{\GG}(s_1,s_2,s_1\chi):=\oo_{\GG_1}(s_1)\otimes\oo_{\GG_2}(s_2)\otimes\oo_{\GG_3}(s_1\chi)$, a very ample line bundle. For each triple $(s_1,s_2,\chi)$ the natural $\GL(V)$-action on $\oo_{\GG}(s_1,s_2,s_1\chi)$ determines a linearization of the $\GL(V)$-action on $\GG$ --- and, therefore, a linearization of the $\GL(V)$-action on $\cc$. 

Note that, for each 1-parameter subgroup $\lambda:\C^*\to \GL(V)$ and every point $t=(t_1,t_2,t_3)\in\GG$, 
\begin{equation*}
    \mu(t,\lambda)=s_1\mu(t_1,\lambda)+s_2\mu(t_2,\lambda)+s_1\chi\mu(t_3,\lambda),
\end{equation*}
where $\mu$ is defined by taking representatives of $t,t_1,t_2$ and $t_3$ in the respective vector spaces and using (\ref{eq-def-mu}).
Note also that, by Remark \ref{remark-properties-mu}, if $t_3=[\Psi',\Psi'']\in\GG_3$ for some $\Psi'\in V_{\Gamma}$ and $\Psi''\in V^\vee_{\Gamma}$, then $\mu(t_3,\lambda)=\max\left(\mu(\Psi',\lambda),\mu(\Psi'',\lambda)\right)$. If we further assume that $t\in\Gies(\cc_0)$, then $\mu(\Psi'',\lambda)=\mu(\Psi'^{*-1},\lambda)$ and we may argue that $\mu(t_3,\lambda)\ge 0$ as in the proof of Lemma \ref{lemma-mu-f-ge-0}, with equality if and only if $\lambda$ is $\Gamma$-invariant with respect to the $\Gamma$-equivariant structure on $V\otimes\oo_X$ determined by $\Psi'$. 

Given $\delta>0$, set
\begin{equation}\label{eq-def-epsilon}
    \epsilon:=\frac{s_1}{s_2}=\frac{\dim V-a\delta}{n\delta}.
\end{equation}
Similarly to the proof of Lemma \ref{lemma-delta,chi-semistable-finite-set}, there exists a finite set $S$ of 1-parameter subgroups $\lambda:\C^*\to \GL(V)$ such that $t=(t_1,t_2,t_3)\in\GG$ if (semi)stable if and only if $\mu(g\cdot t,\lambda)\geqp0$ for every $g\in\GL(V)$ and $\lambda\in S$. We may further assume that all the elements in $S$ have integral weights. 

Since the weights of each $\lambda\in S$ are bounded, given a fixed $\lambda\in S$, $\mu(g\cdot t_i,\lambda)$ is bounded for each $\lambda\in S$ and $i=1,2,3$. In particular, we may choose $\chi\gg\epsilon^{-1}\gg0$ so that, for every 1-parameter subgroup in $S$ and every point $t\in\GG$, $\mu(t,\lambda)>0$ if $\mu(t_3,\lambda)>0$ for each $\lambda\in S$. Therefore, if $t\in\Gies(\cc_0)$ and $\lambda\in S$, then $\mu(t,\lambda)>0$ if and only if either $\mu(t_3,\lambda)>0$, or $\mu(t_3,\lambda)=0$ and $s_1\mu(t_1,\lambda)+s_2\mu(t_2,\lambda)>0$, with $\mu(t_3,\lambda)=0$ if and only if $\lambda$ is $\Gamma$-invariant with respect to the $\Gamma$-equivariant structure defined by $t_3$. In particular, $\mu(t,\lambda)\geqp0$ for every 1-parameter $\lambda$ if and only if $\mu(t,\lambda)\geqp0$ for every $\Gamma$-invariant $\lambda$.

With the above linearization, the Gieseker map allows us to study the GIT quotient $\cc_0\sslash\GL(V)$ by means of the GIT quotient $\GG\sslash\GL(V)$, where $\GG$ is a projective space.

\subsection{Hilbert--Mumford criterion vs (semi)stability}\label{section-hilbert-mumford}

Denote by $\mathfrak G(m_1)$ the set of isomorphism classes of vector bundles $E$ of rank $n$ and trivial determinant for which there exist $m\ge m_1$ and $t\in\cc_0$ such that $E\cong \f_{\cc}\vert_{\{t\}\times X}$ and $\Gies(t)$ is semistable with respect to the linearization in Section \ref{section-gieseker}. In what follows we assume that $\chi\gg\epsilon^{-1}\gg0$ or, alternatively, $\chi\gg\delta\gg0$ --- see Proposition \ref{prop-semistable-pseudo-vs-pairs} and the end of Section \ref{section-gieseker} on why we need these assumptions.

\begin{lemma}\label{lemma-m1}
    There is $m_1>0$ and $C\in\Q$, such that the following holds true: For every $E\in\mathfrak G(m_1)$, and every $\Gamma$-invariant vector subbundle $F\subset E$ --- with respect to the $\Gamma$-equivariant structure given by a point $t\in\cc_0$ corresponding to $E$ ---, $\mu(F)\le C$. 
\end{lemma}
\begin{proof}
    Fix a vector bundle $E\in\mathfrak G$ with corresponding point $t\in\cc_0$. Following the proof of \cite[2.3.5.12]{schmitt}, we may find a lower bound for 
    \begin{equation*}
    \mumin:=\min\{\mu(E/F)\suhthat F\subset E\; \text{is a $\Gamma$-invariant subbundle}\}. 
    \end{equation*}
    Indeed, fix a $\Gamma$-invariant subbundle $F\subset E$. There is an exact sequence
    \begin{equation*}
        0\to H^0(F(m))\to H^0(E(m))\to H^0(Q(m)),
    \end{equation*}
    where $Q:=E/F$. Take a 1-parameter subgroup $\lambda:\C^*\to \SL(V)$ whose corresponding weighted filtration is
    \begin{equation*}
        \left(V_{\bullet}(\lambda):\{0\}\subset V_1:=H^{0}(q(m))^{-1}\left(H^0(F(m))\right)\subset V,\;\alpha_{\bullet}(\lambda)=(1)\right).
    \end{equation*}
    Since $F$ is $\Gamma$-invariant, so is $V_1$, the weighted filtration $V_{\bullet}(\lambda)$ and $\lambda$ itself. Hence, if $\Gies(t)=(t_1,t_2,t_3)$, then $\mu(t_3,\lambda)=0$. In particular, $\mu(\Gies(t),\lambda)=s_1\mu(t_1,\lambda)+s_2\mu(t_2,\lambda)$. Since $\Gies(t)$ is semistable, $s_1\mu(t_1,\lambda)+s_2\mu(t_2,\lambda)\ge0$. Applying the same calculation as in the proof of \cite[2.3.5.12]{schmitt}, we conclude that $\mumin(E)\ge -\frac{\delta a}n-g$, as required.
\end{proof}

\begin{lemma}\label{lemma-m2}
    There exists a positive integer $m_2\ge m_1$ and a finite set
    \begin{equation*}
        S\subset \{(n_{\bullet},\alb)\suhthat n_j\in\N,\alpha_j\in\Q_{\ge 0},\sum_jn_j=n\},
    \end{equation*}
    such that any $m\ge m_2$ satisfies the following. For every pseudo-equivariant pair $(E,\sigma,\phi)$ such that $E\in\mathfrak G(m_2)$, the following conditions are equivalent.
    \begin{enumerate}
        \item $(E,\sigma,\phi)$ is $(\delta,\chi)$-(semi)stable.
        \item For every $\Gamma$-invariant weighted filtration $(\eb,\alb)$ with $(\rk\eb,\alb)\in S$ such that $E_j(m)$ is globally generated and $h^1(E_j(m))=0$ for each $i$, 
        \begin{equation*}
            \sum_{j}\alpha_j\left(h^0(E(m))\rk(E_j)-h^0(E_j(n))\rk(E)\right)+\delta\mu(\eb,\alb,\sigma)\geqp 0.
        \end{equation*}
    \end{enumerate}
\end{lemma}
\begin{proof}
    
    By Lemma \ref{lemma-delta,chi-semistable-finite-set}, there is a finite set
    \begin{equation*}
        S\subset \{(n_{\bullet},\alb)\suhthat n_j\in\N,\alpha_j\in\Q_{\ge 0},\sum_jn_j=n\}
    \end{equation*}
    such that the $(\delta,\chi)$-(semi)stability of any pseudo-equivariant pair $(E,\sigma,f)$ 
    only has to be verified for weighted filtrations $(\eb,\alb)$ such that $(\rk\eb,\alb)\in S$. Since $\chi\gg\delta\gg 0$, this only has to be verified for $\Gamma$-invariant weighted filtrations $(\eb,\alb)$ such that $(\rk\eb,\alb)\in S$. Fix a constant $C'$. Then, there exists $C''$ such that, for every pseudo-equivariant pair $(E,\sigma,f)$ with $E\in\mathfrak G(m_1)$ and every $\Gamma$-invariant weighted filtration $(\eb,\alb)$ such that $(\rk\eb,\alb)\in S$, if
    \begin{equation}\label{eq-condition-mu(Ei)}
        \mu(E_j)\le C''
    \end{equation}
    for some $i$, then $M(\eb,\alb)>C'$. This follows directly from $\mu(E_j)\le\mumaxg(E)$ and the uniform bound on $\mumaxg(E)$ given by Lemma \ref{lemma-m1}, where $\mumaxg(E)$ is the maximum slope attained by a $\Gamma$-invariant vector subbundle of $E$. Using \cite[Lemma 1.5.1.41]{schmitt}, we may determine a constant $C'''$
    such that $\mu(\eb,\alb,\sigma)\ge -C'''$ for every $\Gamma$-invariant weighted filtration $(\eb,\alb)$ such that $(\rk\eb,\alb)\in S$. If we choose $C'\ge\delta C'''$, then 
    \begin{equation*}
        M(\eb,\alb)+\delta\mu(\eb,\alb,\sigma)>C'-\delta C'''\ge0.
    \end{equation*}
    Therefore, $(\delta,\chi)$-(semi)stability only has to be verified for $\Gamma$-invariant weighted filtrations $(\eb,\alb)$ such that $(\rk\eb,\alb)\in S$ and (\ref{eq-condition-mu(Ei)}) fails. By a slightly refined version of \cite[Proposition 2.2.3.7]{schmitt}, these weighted filtrations live in bounded families, hence the lemma follows by \cite[Proposition 2.2.3.2]{schmitt}.
\end{proof}

\begin{proposition}\label{prop-m3}
    There is $m_3\ge m_2$ such that, for every $m\ge m_3$ and every point $t\in\cc_0$ with (semi)stable Gieseker point $\Gies(t)\in\GG$, the corresponding pseudo-equivariant pair $(E_t,\sigma_t,f_t)$ is $(\delta,\chi)$-semistable.
\end{proposition}
\begin{proof}
    We just need to check condition (2) in the statement of Lemma \ref{lemma-m2}, applied to $(E_t,\sigma_t,f_t)$. To do so, follow the same argument as in the proof of \cite[Theorem 2.3.5.13]{schmitt} after the statement of \cite[Corollary 2.3.5.14]{schmitt}.
\end{proof}

Recall that, by construction, for every $(\delta,\chi)$-(semi)stable pseudo-equivariant pair $(E,\sigma,f)$ corresponding to a twisted equivariant $G$-bundle via Propositions \ref{prop-equivalence-of-categories-twisted-pseudo} and \ref{prop-pseudo-are-pairs}, there exists $t\in\cc_0$ such that $(E,\sigma,f)\cong(\f_{\cc_0}\vert_{\{t\}\times X},\phi_{\cc_0}\vert_{\{t\}\times X},\psi_{\cc_0}\vert_{\{t\}\times X})$ --- see Section \ref{section-parameter-space-pseudo-equivariant-pairs}.

\begin{lemma}\label{lemma-m4}
    There exists a positive integer $m_4$, such that any $(\delta,\chi)$-(semi)stable pseudo-equivariant pair $(E,\sigma,f)$ satisfies
    \begin{equation*}
        \sum_{j}\alpha_j(\dim V\cdot \rk E_j-h^0(E_j(m))\cdot n)+\delta\cdot\mu(\eb,\alb,\sigma)\geqp0
    \end{equation*}
    for every $\Gamma$-invariant weighted filtration $(\eb,\alb)$ of $E$ and every $m\ge m_4$ .
\end{lemma}
\begin{proof}
    By the proof of \cite[Proposition 2.3.5.15]{schmitt}, the lemma follows if we can prove that $M(\eb^A,\alb^A)+\delta\mu(\eb^A,\alb^A,\sigma)\geqp0$ for an arbitrary $\Gamma$-invariant weighted filtration $(\eb,\alb)$ --- indeed, the last inequality of the last computation in the proof is the only place where the (semi)stability assumptions are used.
    Note that $\Gamma$-invariance of $(\eb,\alb)$ implies, in particular, that $(\eb^A,\alb^A)$ is $\Gamma$-invariant. By Lemma \ref{lemma-mu-f-ge-0} and Definition \ref{def-delta,chi-semistable-pairs}, the $(\delta,\chi)$-(semi)stability of $(E,\sigma,f)$ implies that $M(\eb^A,\alb^A)+\delta\mu(\eb^A,\alb^A,\sigma)\geqp0$, as required.
\end{proof}

\begin{proposition}\label{prop-m5}
    There exists a positive integer $m_5$ satisfying the following property. If $m\ge m_5$ and $(E,\sigma,f)$ is a $(\delta,\chi)$-(semi)stable pseudo-equivariant pair then, given $t\in \cc_0$ such that $(E,\sigma,f)\cong(\f_{\cc}\vert_{\{t\}\times X},\phi_{\cc}\vert_{\{t\}\times X},\psi_{\cc}\vert_{\{t\}\times X})$, the associated Gieseker point $\Gies(t)$ is (semi)stable for the linearization given in Section \ref{section-gieseker}.
\end{proposition}
\begin{proof}
    Since $\chi\gg\epsilon^{-1}\gg0$, it is enough to check that $s_1\mu(t_1,\lambda)+s_2\mu(t_2),\lambda)\geqp0$ for all non-trivial $\Gamma$-invariant 1-parameter subgroups $\lambda$. To do this, just copy the proof of \cite[Theorem 2.3.5.16]{schmitt} and replace the arbitrary $\lambda$ with a $\Gamma$-invariant $\lambda$, and \cite[Proposition 2.3.5.15]{schmitt} by Lemma \ref{lemma-m4}. This works because the $\Gamma$-invariance of $\lambda$ clearly implies that the associated filtration $\eb$ is $\Gamma$-invariant.
\end{proof}

\begin{theorem}\label{th-semistability-vs-hilbert-mumford}
    There exists a positive integer $m_0$ satisfying the following property. If $m\ge m_0$ and $\chi\gg\delta\gg0$, then, given $t\in \cc_0$ and $(E,\sigma,f):=(\f_{\cc}\vert_{\{t\}\times X},\phi_{\cc}\vert_{\{t\}\times X},\psi_{\cc}\vert_{\{t\}\times X})$, the pseudo-equivariant pair $(E,\sigma,f)$ is $(\delta,\chi)$-(semi)stable if and only if $\Gies(t)$ is (semi)stable for the linearization given in Section \ref{section-gieseker}.
\end{theorem}

\begin{corollary}\label{cor-parameter-space-semistable-twisted-equivariant}
    There exists a complex quasi-projective variety $\ccss$, together with a universal $(\theta,c)$-twisted $\Gamma$-equivariant $G$-bundle $(\mathcal Q,\action)$ over $\ccss\times X$ such that, for every semistable $(\theta,c)$-twisted $\Gamma$-equivariant $G$-bundle $(\bundle,\action)$ over $X$, there exists $t\in\ccss$ with $(\bundle,\action)\cong (\mathcal Q\vert_{\{t\}\times X},\action)$. Moreover, $\ccss$ features an algebraic $\GL(V)$-action lifting to a $\GL(V)$-action on $\mathcal Q$, such that $(\mathcal Q\vert_{\{t\}\times X},\action)\cong(\mathcal Q\vert_{\{t'\}\times X},\action)$ if and only if $t$ and $t'$ are in the same $\GL(V)$-orbit.
\end{corollary}
\begin{proof}
    By Theorem \ref{th-semistability-vs-hilbert-mumford}, the subset $\ccss\subset\cc_0$ consisting of $(\delta,\chi)$-semistable pseudo-equivariant pairs is an open subvariety. The result follows from Propositions \ref{prop-iso-vs-orbit}, \ref{prop-semistable-pseudo-vs-pairs} and \ref{prop-equivalence-of-categories-twisted-pseudo}.
\end{proof}

\subsection{The moduli space is a projective variety}\label{section-projective}

Let $\GGss\subset\GG$ be the open subvariety of semistable points in $\GG$ in the sense of GIT, with respect to the $\GL(V)$-action and the linearization given in Section \ref{section-gieseker}. Choose $\chi\gg\delta\gg0$ and $\chi\gg\epsilon^{-1}\gg0$ so that Propositions \ref{prop-m3} and \ref{prop-m5} hold. Let $\ccss\subset\cc_0$ be the subvariety consisting of points corresponding to $(\delta,\chi)$-semistable pseudo-equivariant pairs. Pick $m_0$ as in Theorem \ref{th-semistability-vs-hilbert-mumford}. Then, for $m\ge m_0$, the Gieseker morphism restricts to an injection $\Giesss:\ccss\hookrightarrow\GGss$. 

\begin{proposition}\label{prop-properness}
    The Gieseker morphism  $\Giesss:\ccss\hookrightarrow\GGss$ is proper.
\end{proposition}
\begin{proof}
    We use the valuative criterion for properness. Let $R$ be a discrete valuation ring with quotient field $K$ and set $C:=\Spec(R)$ and $C^*:=\Spec(K)$. Let $\eta:C\to\GGss$ be a morphism such that the restriction $\eta\vert_{C^*}:C^*\to\GGss$ lifts to a morphism $\eta^*:C^*\to\ccss$. We have to show that $\eta^*$ extends to a morphism $C\to\ccss$ lifting $\eta$. 

    Since $\ccss$ is a parameter space for the set of $(\delta,\chi)$-semistable pseudo-equivariant pairs --- see Section \ref{section-parameter-space-pseudo-equivariant-pairs} ---, the morphism  $\eta^*$ is associated to a family of $(\delta,\chi)$-semistable pseudo-equivariant pairs
    \begin{equation*}
        (q_{C^*}:V\otimes\pi_X^*(\oo_X(-m))\otimes\pi_{C^*}^*\oo_{C^*}\to E_{C^*},\sigma_{C^*},f_{C^*}).
    \end{equation*}
    Here $E_{C^*}\to C^*\times X$ is a family of vector bundles of rank $n$ and trivial determinant, $\sigma_{C^*}$ is a homomorphism
    \begin{equation*}
        \sigma_{C^*}:(E_{C^*})_{a,b}=(E^{\otimes a})^{\oplus b}\to(\det E)^{\otimes c}\otimes\pi_X^*\oo_X,
    \end{equation*}
    and $(\id_{C^*}\times\gamma)^*E_{C^*}\}_{\gamma\in\Gamma}$ is a family of isomorphisms satisfying \ref{eq-pseudo-equivariant-composition-f} and \ref{eq-pseudo-equivariant-pairs-f-sigma}. 
    
    Recall that the Quot Scheme $\qq$ has a natural compactification $\oqq$ parametrizing coherent sheaves of rank $n$. Therefore, $q_{C^*}$ extends to a family $q_C:V\otimes\pi_X^*(\oo_X(-m))\to E_{C}$, where the restriction $E_C\vert_{\{0\}\times X}$ to the special fibre may have torsion at a subset of points $Z\subset X$. However, note that $\widetilde E_C:=E_C^{**}$, is a family of locally-free sheaves, since $\widetilde E_C$ is a reflexive sheaf over a the regular two-dimensional scheme $C\times X$ --- see \cite[Corollaries 1.3 and 1.4]{reflexive}. 
    
    By the proof of \cite[Proposition 2.3.5.17]{schmitt}, $\sigma_{C^*}$ extends to a family of homomorphisms $\sigma_C:(E_{C})_{a,b,c}\to\pi_X^*\oo_X$. As in such proof, let $\iota:\cxz\hookrightarrow C\times X$ be the open embedding and let
    \begin{align*}
        \tsigma_{C}:(\tE_{C})_{a,b}\to\iota_*\left((\tE_{C})_{a,b}\vert_{\cxz}\right)=
        \iota_*\left((E_{C})_{a,b}\vert_{\cxz}\right)\\
        \xrightarrow{\iota_*(\sigma_C\vert_{\cxz})}
        \iota_*\left(\det(E_C\vert_{\cxz})^{\otimes c}\otimes\pi_X^*\oo_X\vert_{\cxz}\right)=
        \det(\tE_C)^{\otimes c}\otimes\pi_X^*\oo_X.
    \end{align*}
    
    Consider the projection $\eta_3:C\to\GG_3=\PP(V_{\Gamma}\oplus V^\vee_{\Gamma})$ to the third component of $\GG$. \textbf{Claim}: the point $\eta_3(0)\in\PP(V_{\Gamma}\oplus V^\vee_{\Gamma})$ is equal to the class of $\bigoplus_{\gamma\in\Gamma}f_{0,\gamma}\oplus \bigoplus_{\gamma\in\Gamma}f^{*-1}_{0,\gamma}$ in projective space, where $f_0$ is a family of isomorphisms satisfying (\ref{eq-pseudo-equivariant-composition-f}) and (\ref{eq-pseudo-equivariant-pairs-f-sigma}) --- recall the notation (\ref{eq-*-1}). Indeed, $\eta_3(C^*)$ is equal to the image of the family $\bigoplus_{\gamma\in\Gamma}f_{C^*,\gamma}\oplus \bigoplus_{\gamma\in\Gamma}f^{*-1}_{C^*,\gamma}$ in projective space, which we denote by $\left[\bigoplus_{\gamma\in\Gamma}f_{C^*,\gamma}\oplus \bigoplus_{\gamma\in\Gamma}f^{*-1}_{C^*,\gamma}\right]$. This has an extension to a family $\left[\bigoplus_{\gamma\in\Gamma}f_{C,\gamma}\oplus \bigoplus_{\gamma\in\Gamma}f^{*-1}_{C,\gamma}\right]$, where $f_{0,\gamma}$ may not be an isomorphism for some $\gamma$. This, in turn, induces a family $\left[\bigoplus_{\gamma\in\Gamma}\tf_{C,\gamma}\oplus \bigoplus_{\gamma\in\Gamma}\tf^{*-1}_{C,\gamma}\right]$, where $\tf_{C,\gamma}:\tE_C\to(\id_C\times\gamma)^*\tE_C$. 
    
    Pick $\gamma\in\Gamma$. Since (\ref{eq-pseudo-equivariant-composition-f}) is a closed condition,
    \begin{equation*}
    (\gamma^{*(\vert\gamma\vert-1)}\tf_{0,\gamma})\circ (\gamma^{*(\vert\gamma\vert-2)}\tf_{0,\gamma})\circ\dots\circ(\gamma^*\tf_{0,\gamma})\circ\tf_{0,\gamma}=\tf_{0,1}=\id.    
    \end{equation*}
    In particular, $\tf_{0,\gamma}$ is an isomorphism for every $\gamma\in\Gamma$. Since (\ref{eq-pseudo-equivariant-pairs-f-sigma}) is a closed condition, the claim is proved.
    

    Finally, as in the proof of \cite[Proposition 2.3.5.17]{schmitt}, one can check that $(\tE_0,\tsigma_0,\tf_0)$ is a $(\delta,\chi)$-semistable pseudo-equivariant pair. Since its image in $\GGss$ is equal to $\eta(0)$, it determines a extension $\eta^*:C\to\ccss$ lifting $\eta$, as required.
\end{proof}

\subsection{S-equivalence vs polystability}\label{section-polystable}

Recall that the GIT quotient $\ccss\sslash\GL(V)$ parametrizes S-equivalence classes of semistable $(\theta,c)$-twisted $\Gamma$-equivariant $G$-bundles over $X$. Here, two points of $\ccss$ are said to be \textbf{S-equivalent} if the closure of their $\GL(V)$-orbits intersect. In the closure of each orbit there is a unique closed orbit, so we may also see $\ccss\sslash\GL(V)$ as the moduli space of isomorphism classes of $(\theta,c)$-twisted $\Gamma$-equivariant $G$-bundles over $X$ whose corresponding $\GL(V)$-orbit is closed. The goal of this section is to show that these are precisely the polystable ones, according to Definition \ref{def-semistable-twisted-equivariant}.

\begin{lemma}\label{lemma-mu=0}
    Assume that $\chi\gg\delta\gg0$, as usual. Let $t\in\ccss$ with corresponding pseudo-equivariant pair $(E_t,\sigma_t,f_t)$. Let $\lambda:\C^*\to\GL(V)$ be a non-trivial 1-parameter subgroup with corresponding weighted filtration $(\ebt,\alb)$. Then, $\mu(\Gies(t),\lambda)=0$ --- according to the linearization given by $\oo(s_1,s_2,s_1\chi)$, see Section \ref{section-gieseker} --- if and only if the corresponding weighted filtration $(\ebt,\alb)$ satisfies $M(\ebt,\alb)=\mu(\ebt,\alb,\sigma_t)=\mu(\ebt,\alb,f_t)=0$.
\end{lemma}
\begin{proof}

    Let $\Gies(t)=(t_1,t_2,t_3)\in\GG=\GG_1\times\GG_2\times\GG_3$. Recall from Section \ref{section-gieseker} that
    \begin{equation*}
        \mu(\Gies(t),\lambda)=s_1\mu(t_1,\lambda)+s_2\mu(t_2,\lambda)+s_1\chi\mu(t_3,\lambda),
    \end{equation*}
    where $s_1$ and $s_2$ satisfy (\ref{eq-def-epsilon}) and $\chi\gg\delta\gg 0$. 

    First assume that $\mu(\Gies(t),\lambda)=0$. As in the proof of Lemma \ref{lemma-delta,chi-semistable-finite-set}, there exists a finite set of weights $\Lambda$ --- independent of $t$ --- and $g\in\GL(V)$ --- depending on $t$ --- such that
    \begin{equation}\label{eq-decomposition-mu-lambda}
        \mu(\Gies(t),\lambda)
        =\sum_{\lambda'\in\Lambda}l_{\lambda'}\mu(\Gies(t),g\lambda' g^{-1})
        ,\nonumber
    \end{equation}
    for some $l_{\lambda'}\ge0$. Since $\Gies(t)$ is semistable, $\mu(\Gies(t),g\lambda' g^{-1})\ge0$ for every $\lambda'\in\Lambda$. In particular, for each $\lambda'\in\Lambda$, either $l_{\lambda'}=0$ or
    \begin{equation*}
        \mu(\Gies(t),g\lambda' g^{-1})=s_1\mu( t_1,g\lambda' g^{-1})+s_2\mu( t_2,g\lambda' g^{-1})+s_1\chi\mu( t_3,g\lambda' g^{-1})=0.
    \end{equation*}
    Since $\chi\gg\delta\gg0$, i.e. $\chi\gg\epsilon^{-1}\gg0$, by (\ref{eq-def-epsilon}) this implies that $\mu( t_3,g\lambda' g^{-1})=s_1\mu( t_1,g\lambda' g^{-1})+s_2\mu( t_2,g\lambda' g^{-1})=0$. Therefore, plugging this back into (\ref{eq-decomposition-mu-lambda}),
    \begin{equation*}
        \mu( t_3,\lambda )=0=s_1\mu( t_1,\lambda )+s_2\mu( t_2,\lambda).
    \end{equation*}
    Because of the first equation, the associated filtration $V_{\bullet}$ is $\Gamma$-invariant, and so is the corresponding filtration $\ebt$, which implies that $\mu(\ebt,\alb,f_t)=
0$. 

We claim that $s_1\mu(t_1,\lambda)+s_2\mu(t_2,\lambda)=0$ implies that $M(\ebt,\alb)+\delta\mu(\ebt,\alb,\sigma_t)=
0$. This is just the same calculation as for the moduli space of principal $G$-bundles with no twisted equivariant structure. For example, use a similar calculation to the one in the proof of \cite[Theorem 2.3.5.13]{schmitt} after the statement of \cite[Corollary 2.3.5.14]{schmitt}. 

Next, using Lemma \ref{lemma-delta,chi-semistable-finite-set}, we may find a finite set $S$ and a decomposition
\begin{equation*}
    M(\ebt,\alb)+\delta\mu(\ebt,\alb,\sigma_t)=\sum_{(\rk\eb',\rk\alb')\in S}l_{(\eb',\alb')}\left(M(\ebt',\alb')+\delta\mu(\ebt',\alb',\sigma_t)\right),
\end{equation*}
with $l_{(\eb',\alb')}\ge 0$ and only finitely many $l_{(\eb',\alb')}> 0$. Since $(E_t,\sigma_t,f_t)$ is semistable, either $l_{(\eb',\alb')}=0$ or $M(\ebt',\alb')+\delta\mu(\ebt',\alb',\sigma_t)=0$. Therefore, for $\delta\gg0$, either $l_{(\eb',\alb')}=0$ or $M(\ebt',\alb')=\mu(\ebt',\alb',\sigma_t)=0$. This, in turn, implies that
\begin{equation*}
    M(\ebt,\alb)=\mu(\ebt,\alb,\sigma_t)=0,
\end{equation*}
as required.

Now assume that $M(\ebt,\alb)=\mu(\ebt,\alb,\sigma_t)=\mu(\ebt,\alb,f_t)=
0$. In particular $\ebt$ is $\Gamma$-invariant, and so is the associated filtration $V_{\bullet}$, which in turn yields $\mu(t_3,\lambda)=0$. Hence, we have to show that $M(\ebt,\alb)+\delta\mu(\ebt,\alb,\sigma_t)=
0$ implies $s_1\mu(t_1,\lambda)+s_2\mu(t_2,\lambda)=0$. This is the same calculation as for the moduli space of principal $G$-bundles with no twisted equivariant structure. For example, use a similar calculation to the one in the proof of \cite[Theorem 2.3.5.16]{schmitt}.
\end{proof}

\begin{proposition}\label{prop-polystability-vs-S-equivalence}
    A semistable $(\theta,c)$-twisted $\Gamma$-equivariant $G$-bundle $(\bundle,\action)$ over $X$ is poly-stable if and only if the corresponding point of $\ccss$ has closed $\GL(V)$-orbit. 
\end{proposition}
\begin{proof}
    Let $t\in\ccss$ be the point corresponding to $(\bundle,\action)$, with corresponding pseudo-equivariant pair $(E_t,\sigma_t,f_t)$. For each $t'\in\overline{\GL(V)\cdot t}\subset\ccss$, denote the corresponding pseudo-equivariant pair by $(E_{t'},\sigma_{t'},f_{t'})$. 
    By Proposition \ref{prop-iso-vs-orbit}, we have to show that $(E_t,\sigma_t,f_t)$ is polystable if and only if $(E_t,\sigma_t,f_t)\cong (E_{t'},\sigma_{t'},f_{t'})$ for every $t'\in\overline{\GL(V)\cdot t}$.
    
    By the Hilbert--Mumford criterion, there exists a non-trivial 1-parameter subgroup $\lambda:\C^*\to\GL(V)$ such that $\lim_{z\to\infty}\lambda(z)\cdot t\in\GL(V)\cdot t'$ and $\mu(\Gies(t),\lambda)=0$, where $\Gies$ is the Gieseker map --- see Section \ref{section-gieseker}. Let $(\ebt,\alb)$ be the weighted filtration determined by $\lambda$. By Lemma \ref{lemma-mu=0}, $M(\ebt,\alb)=\mu(\ebt,\alb,\sigma_t)=\mu(\ebt,\alb,f_t)=
0$. Since $\chi\gg\delta\gg0$, $\mu(\ebt,\alb,f_t)=\mu(\ebt,\alb,\sigma_t)=M(\ebt,\alb)=0$. In particular, $\eb$ is $\Gamma$-invariant and the $\C^*$-action on $E_t$ given by $\lambda$ respects the reduction of $\hhom(\oo^{\oplus n},E)$ to $G$ given by $\sigma_t$. 

Therefore, $\lambda$ defines a non-trivial 1-parameter subgroup $\olambda$ of a $\Gamma$-invariant maximal torus $T$ of $G$ which, after possibly replacing $t$ with another element in its $\GL(V)$-orbit, we may assume to be independent of $\lambda$. Since $\lambda$ is $\Gamma$-invariant, so is $\olambda$. Therefore, $\olambda$ is generated by some $s\in i\lie t^{\Gamma}$, which in turn defines parabolic subgroup $P_s\subset G$ and $\tP_s\subset\GL(n,\C)$. 
The decomposition of $E$ into weight subspaces with respect to the $\C^*$-action given by $\lambda$ defines the weighted filtration $(\eb,\alb)$, which in turn determines a reduction of structure group $\ttau\in H^0\left(X,\hhom(\oo^{\oplus n},E)(\GL(n,\C)/\tP_s)\right)$ of the frame bundle. This, in turn, determines a reduction $\tau\in H^0(X,\bundle(G/P_s))$. Moreover, $\deg E(\tau,s)=M(\ebt,\alb)=0$. 

It can be seen that $\lim_{z\to\infty}\lambda(z)\cdot E_t=\bigoplus_{j}E_j/E_{j-1}$, with induced $G$-bundle structure $\lim_{z\to\infty}\lambda(z)\cdot \sigma_t$ and equivariant structure $\lim_{z\to\infty}\lambda(z)\cdot f_t$. The frame bundle of $\lim_{z\to\infty}\lambda(z)\cdot E_t$ is equal to $\hhom(\oo^{\oplus n},E)_{\ttau}(\tL_s)$, where $\hhom(\oo^{\oplus n},E)_{\ttau}$ is the $\tP_s$-bundle over $X$ associated to $\ttau$, and $\tL_s$ is the Levi subgroup of $\GL(n,\C)$ determined by $s$. Therefore, the associated $G$-bundle by $\lim_{z\to\infty}\lambda(z)\cdot \sigma_t$ is equal to the extension of structure group of $Q_{\tau}(L_s)$ to $G$. 

We conclude that $(E_t,\sigma_t,f_t)\cong (E_{t'},\sigma_{t'},f_{t'})$ holds if and only if there is a $\Gamma$-equivariant isomorphism $Q_{\tau}(L_s)\cong Q$, which is true if and only if $Q_{\tau}$ admits a $\Gamma$-invariant reduction of structure group to $L_s$ --- this last equivalence can be shown using Proposition \ref{prop-iso-vs-orbit}.
From this it is clear that, if $(Q,\action)$ is polystable, then $(E_t,\sigma_t,f_t)\cong (E_{t'},\sigma_{t'},f_{t'})$ for every $t'\in\overline{\GL(V)\cdot t}$. 

Conversely, assume that $(E_t,\sigma_t,f_t)\cong (E_{t'},\sigma_{t'},f_{t'})$ for every $t'\in\overline{\GL(V)\cdot t}$. Take $s\in i\lie t^{\Gamma}\setminus \{0\}$ and $\tau\in H^0(X,\bundle(G/P_s))^{\Gamma}$ such that $\deg E(\tau,s)=0$ --- by Definition \ref{def-semistable-twisted-equivariant}, we have to show that there is a further reduction of structure group of $\bundle_{\tau}$ to $L_s$. The data $\tau, s$ induces a weighted filtration $(\ebt,\alb)$. As in Lemma \ref{lemma-m2}, we may assume that $h^1(X,E_j(m))=0$ and $E_j(m)$ is globally generated for every $j$, therefore $(\ebt,\alb)$ is determined by a non-trivial 1-parameter subgroup $\lambda$ of $\GL(V)$, namely the one associated with the weighted filtration $\left(H^0(q_{\cc}\otimes\id_{\pi_X^*\oo_X(m)})^{-1}(H^0(\ebt(m)),\alb\right)$ of $V$. Set $t':=\lim_{z\to\infty}\lambda(z)\cdot t$. Then, by the first sentence of the previous paragraph, $(E_t,\sigma_t,f_t)\cong (E_{t'},\sigma_{t'},f_{t'})$ implies that $Q_{\tau}$ admits a reduction of structure group to $L_s$, as required.
\end{proof}

\subsection{The moduli space}
\begin{theorem}\label{th-moduli-space-G-semisimple}
    Assume that $G$ is a connected semisimple complex Lie group. Then there exists a complex projective variety $\Mm(X,G,\Gamma,\theta,c)$ which is a coarse moduli space classifying isomorphism classes of polystable $(\theta,c)$-twisted $\Gamma$-equivariant $G$-bundles over $X$. It contains an open subvariety classifying isomorphism classes of stable $(\theta,c)$-twisted $\Gamma$-equivariant $G$-bundles over $X$.
\end{theorem}
\begin{proof}
    By Theorem \ref{th-semistability-vs-hilbert-mumford} and Propositions \ref{prop-semistable-pseudo-vs-pairs} and \ref{prop-equivalence-of-categories-twisted-pseudo}, we may choose $m\ge m_0$, such that the linearization given in Section \ref{section-gieseker} provides a GIT quotient $\Mm(X,G,\Gamma,\theta,c):=\ccss\sslash\GL(V)$ parametrizing S-equivalence classes of semistable $(\theta,c)$-twisted $\Gamma$-equivariant $G$-bundles over $X$. By Proposition \ref{prop-polystability-vs-S-equivalence}, $\Mm(X,G,\Gamma,\theta,c)$ classifies isomorphism classes of polystable $(\theta,c)$-twisted $\Gamma$-equivariant $G$-bundles over $X$. By Proposition \ref{prop-properness}, it is a complex projective variety.
\end{proof}

\section{The moduli space for \texorpdfstring{$G$}{G} reductive}\label{section-moduli-space-reductive}

Fix the data $X, G,\Gamma,\theta$ and $c$ as in previous sections, where $G$ is now an arbitrary connected reductive complex Lie group. Recall that the construction in Sections \ref{section-parameter-space} and \ref{section-moduli-space-semisimple} only works when $G$ is semisimple. However, we may use it to build the moduli space of $(\theta,c)$-twisted $\Gamma$-equivariant $G$-bundles for reductive $G$. We follow the approach in \cite{gomez-sols}, which is a generalization of \cite{ramanathan1}. 

We include here a result that will be needed throughout this section. Recall that, in this general setting, the (poly, semi)stability notions depend on a parameter, which is a $\Gamma$-invariant element $\cha\in i\lie z_{\lie k}^{\Gamma}$ --- here $\lie z_{\lie k}$ is the centre of a $\Gamma$-invariant maximal compact subalgebra $\lie k$ of $\lie g$, with corresponding maximal compact subgroup $K$ ---, and a $G$-invariant pairing $\lie g\cong\lie g^*$.

Recall that the topology of a $G$-bundle over $X$ can be identified with an element of $Z(G)^*$, or $\lie z^*$, where $Z(G)$ is the centre of $G$ with Lie algebra $\lie z$. Via the above pairing, this may be regarded as an element of $\lie z$.

\begin{definition}\label{def-topological-type}
    The \textbf{topological type} of a $G$-bundle $\bundle$ is the number
$\frac{i}{2\pi}\int_XF$, where $F$ is the Chern curvature corresponding to some smooth reduction of structure group of $\bundle$ to the maximal compact subgroup $K$. 
\end{definition}


\begin{proposition}\label{prop-topology-parameter}
    Take $\cha\in i\lie z_{\lie k}^{\Gamma}$ and let $(\bundle,\action)$ be a $\cha$-polystable $(\theta,c)$-twisted $\Gamma$-equivariant $G$-bundle over $X$. Then $\bundle$ has topological type $\cha$. Conversely, if a $G$-bundle of topological type $\cha$ is $\cha_0$-polystable for some $\cha_0\in i\lie z_{\lie k}^{\Gamma}$, then $\cha_0=\cha$.
\end{proposition}
\begin{proof}
    By Proposition \ref{prop-semistability-underlying-bundle}, $\bundle$ is $\cha$-polystable as a $G$-bundle. Therefore, by the Hitchin--Kobayashi correspondence, there exists a smooth reduction of structure group to $K$ whose corresponding Chern curvature $F$ satisfies
    \begin{equation*}
        F=-2\pi i\cha\omega,
    \end{equation*}
    where $\omega$ is a volume form on $X$ with total volume 1. In particular,
    \begin{equation*}
        \frac{i}{2\pi}\int_XF=\cha.
    \end{equation*}
    Since the left hand side is the topological type of $\bundle$, the proposition follows.
\end{proof}

\subsection{Finite torsors and moduli problems}\label{section-torsors}

This section is inspired by \cite{gomez-sols}. We consider two moduli problems, represented by presheaves $\Ff1$ and $\Ff2$ over a site of of schemes over $\C$ --- with a suitable Grothendieck topology --- with values in groupoids. These presheaves represent families of geometric objects over $X$, parametrized by the input scheme. We make the following \textbf{assumptions}.
\begin{enumerate}
    \item There is a natural transformation $\eta:\Ff1\to\Ff2$ and a finite abelian group $A$ such that, for every contractible scheme $U$ --- in the strong analytic topology --- and every $p\in\Ff2(U)$, $\eta^{-1}(p)$ is an $A$-torsor.
    \item For every complex variety $S$, the functor $\eta(S):\Ff1(S)\to\Ff2(S)$ makes $\Ff1(S)$ a fibred category over $\Ff2(S)$.
    \item For every scheme $S/\C$, $\eta(S):\Ff1(S)\to\Ff2(S)$ is surjective on objects.
    \item There exist a complex variety $R_2$ and an element $u_2\in\Ff2(R_2)$ such that any family in the image of $\Ff2$ is locally the pullback of $u_2$ by some morphism to $R_2$. In other words, the sheafification $\tFf2$ of $\Ff2$ is representable by $R_2$.
    \item There exists a complex reductive group $H$ and an algebraic $H$-action on $R_2$ which lifts to an $H$-action on $u_2$ and such that, given two points $r$ and $r'\in R_2$, $u_2\vert_{\{r\}\times X}\cong u_2\vert_{\{r'\}\times X}$ if and only if $r'\in H\cdot r$.
\end{enumerate}


\begin{lemma}\label{lemma-R1-from-R2}
    Let $\tFf2$ be the sheafification of $\Ff2$. Then $\tFf2$ is representable by a complex variety $R_1$ which is an unramified Galois covering $p:R_1\to R_2$ with Galois group $A$. Moreover, the $H$-action on $R_2$ lifts to an $H$-action on $R_1$, such that two points $r$ and $r'\in R_1$ are in the same $H$-orbit if and only if $u_1\vert_{\{r\}\times X}\cong u_1\vert_{\{r'\}\times X}$, where $u_1\in\tFf1(R_1)$ is a universal family. 
\end{lemma}
\begin{proof}
    The construction of $R_1$ is similar to \cite[Theorem 2.4.8.7]{schmitt}.
    Given an open subspace $U$ of $R_2$, we denote by $\Ff2^U$ the functor from the site of schemes over $U$ to sets, sending a morphism $f:S\to U$ to $\eta(S)^{-1}(f^*u_2\vert_{U\times X})$. 

    Since $\tFf2$ is a sheaf,
    it suffices to check that there is an open covering $\curly U=(U_i)_i$ of $R_2$ such that $\Ff2^{U_i}$ is representable by $U_i\times A$ as an $A$-torsor. This follows from Assumption (2). After gluing, we obtain an unramified Galois cover $p:R_1\to R_2$ with Galois group $A$. The universal family $u_1\in\tFf1(R_1)$ is obtained by gluing elements $u_1\vert_{p^{-1}(U_i)\times X}\in\Ff1(p^{-1}(U_i))$. There is a natural isomorphism $\teta(R_1)(u_1)\xrightarrow{\sim}p^*(u_2)$, where $\teta:\tFf1\to\tFf2$ is induced by $\eta$.

    In order to lift the $H$-action, we use the same argument as in the proof of \cite[Proposition 2.4.8.9]{schmitt}. The $H$-action on $R_2$ is given by a morphism
    \begin{equation*}
        \phi_H:H\times R_2\to R_2.
    \end{equation*}
    The lift of this $H$-action to $u_2$, which exists by Assumption (5), is given by an isomorphism
    \begin{equation*}
        \tphi_H:\pi_{R_2}^*u_2\xrightarrow{\sim} (\phi_H\times\id_X)^*u_2
    \end{equation*}
    over $H\times R_2\times X$. Pulling back to $R_1$ provides an isomorphism
    \begin{equation*}
        \tphi'_H:p^*\pi_{R_2}^*u_2\xrightarrow{\sim} p^*(\phi_H\times\id_X)^*u_2
    \end{equation*}
    $H\times R_1\times X$. Via the isomorphism $\teta(R_1)(u_1)\xrightarrow{\sim}p^*(u_2)$ and Assumption (2), $\tphi'_H$ induces an automorphism of $\pi^*_Hu_1$ over $H\times R_1\times X$. This, in turn, induces an $H$-action on $R_1$.
    
    By construction, if two points of $R_1$ are in the same $H$-orbit, then they yield isomorphic objects over $X$. Conversely, if $u_1\vert_{\{r\}\times X}\cong u_2\vert_{\{r'\}\times X}$ for two points $r$ and $r'\in R_1$, then $p^*u_2\vert_{\{r\}\times X}\cong p^*u_2\vert_{\{r'\}\times X}$ and $p(r')\in H\cdot p(r)$ by Assumption (5), which by construction implies that $r'\in H\cdot r$.
\end{proof}

\subsection{Moduli space for \texorpdfstring{$G$}{G} abelian}\label{section-G-abelian}

In this section we address the construction of the moduli space of $(\theta,c)$-twisted equivariant $G$-bundles for $G$ abelian. In other words,  $G\cong (\C^*)^{\times p}$ for some $p\in\N$. In this case
\begin{equation}\label{eq-automorphism-abelian-G}
    \Aut(G)\cong \GL_p(\Z),
\end{equation}
where $\GL_p(\Z)$ acts on $\C^*$ via 
\begin{equation*}
    (a_{ij})_{1\le i,j\le p}\cdot (z_1,\dots,z_p):=(\prod_k z_{k}^{a_{1k}},\prod_k z_{k}^{a_{2k}},\dots,\prod_kz_{k}^{a_{pk}}).
\end{equation*} 

Recall that there is a subgroup $Z^1_{\theta}(\Gamma,G)\subset G^{\Gamma}$ consisting of maps $\kappa:\Gamma\to G$ which satisfy
\begin{equation*}
    \kappa(\gamma_1\gamma_2)=\kappa(\gamma_1)\theta_{\gamma_1}(\kappa(\gamma_2)),
\end{equation*}
called \textbf{1-cocycles}. The group $G$ acts on $Z^1_{\theta}(\Gamma,G)$ by
\begin{equation*}
     Z^1_{\theta}(\Gamma,G)\times G\to Z^1_{\theta}(\Gamma,G);\,\kappa\cdot g:=g^{-1}\kappa({\gamma})\tg(g).
\end{equation*}
The quotient $H^1_{\theta}(\Gamma,G):=Z^1_{\theta}(\Gamma,G)/G$ is the \textbf{first cohomology group} of $\Gamma$ with values in $G$ --- see, for example, \cite{serre-galois}. 

\begin{proposition}\label{prop-G-abelian-twisted-equivariant-actions-H1}
    Given a $G$-bundle $\bundle\to X$, the set of isomorphism classes of $(\theta,c)$-twisted $\Gamma$-equivariant structures on $\bundle$ is either empty or an $H^1_{\theta}(\Gamma,G)$-torsor.
\end{proposition}
\begin{proof}
Given a $G$-bundle $\bundle\to X$ admitting some $(\theta,c)$-twisted $\Gamma$-equivariant structure $\action$, the set of $(\theta,c)$-twisted $\Gamma$-equivariant structures on $\bundle$ is a $Z^1_{\theta}(\Gamma,G)$-torsor. Indeed, any other $(\theta,c)$-twisted $\Gamma$-equivariant structure $*$ is of the form
\begin{equation*}
    q*\gamma=(q\cdot \kappa(\gamma))\action\gamma
\end{equation*}
for each $q\in\bundle$ and $\gamma\in\Gamma$, where $\kappa:\Gamma\to G$ is some map. The condition that $*$ is $(\theta,c)$-twisted --- i.e., the corresponding 2-cocycle $c$ is the same as that of $\action$ --- implies that $\kappa$ is a 1-cocycle in $Z^1_{\theta}(\Gamma,G)$.

Moreover, any automorphism of $\bundle$ is given by multiplication by an element $g\in G$. This fits into the commutative diagramme
\begin{equation*}
    \begin{tikzcd}
        \bundle \arrow[r,"\cdot g"]\arrow[d,"(\cdot\kappa({\gamma}))\bullet\gamma"]    &  \bundle\arrow[d,"(\cdot g^{-1}\kappa({\gamma})\tg(g))\bullet\gamma"]\\
        \bundle \arrow[r,"\cdot g"]   &   \bundle
    \end{tikzcd}
\end{equation*}
for each $\gamma\in\Gamma.$ Therefore, the group of $G$-bundle automorphisms of $\bundle$ acts on $Z^1_{\theta}(\Gamma,G)$ via
\begin{equation*}
     Z^1_{\theta}(\Gamma,G)\times G\to Z^1_{\theta}(\Gamma,G);\,\kappa\cdot g:=g^{-1}\kappa({\gamma})\tg(g).
\end{equation*}
The set of isomorphism classes of $(\theta,c)$-twisted $\Gamma$-equivariant structures on $\bundle$ is thus a torsor of the quotient of $Z^1_{\theta}(\Gamma,G)$ by this $G$-action, which is precisely equal to the first cohomology group $H^1_{\theta}(\Gamma,G)$.
\end{proof}

Fix a character $d\in G^*$. Via the isomorphism $G^*\cong\Z^p$, this can be identified with a tuple of $p$ integers $d=(d_1,\dots,d_p)$. A $G$-bundle $\bundle$ has topological type $d$ --- according to Definition \ref{def-topological-type} --- if and only if it is a product of $p$ $\C^*$-bundles of degrees $d_1,\dots,d_p$. Therefore, in the following we call $d$ the \textbf{degree} of $\bundle$.

The set of isomorphism classes of $G$-bundles of degree $d$ admits a fine moduli space
\begin{equation*}
    \Jac^{d}:=\prod_{1\le k\le p}\Jac^{d_k},
\end{equation*}
where $\Jac^{d_k}$ is the degree $d_k$ part of the Picard group of $X$. A universal $G$-bundle is given by a product of Poincar\'e bundles
\begin{equation}\label{eq-poincare}
    \Pp^d:=\bigotimes_{1\le k\le p}\pi_k^*\Pp^{d_k},
\end{equation}
where $\pi_k:G\to\C^*$ is the projection onto the $k$-th factor.
Let $\Jac^{d,\Gamma}_{\theta,c}\subset\Jac^d$ be the closed subvariety consisting of $G$-bundles of degree $d$ admitting a $(\theta,c)$-twisted $\Gamma$-equivariant structure. This is a connected component of the subvariety $(\Jac^d)^{\Gamma}\subset\Jac$ of $\Gamma$-fixed points, where $\gamma\in\Gamma$ sends a $G$-bundle $\bundle$ to $\gamma^*\tg(\bundle)$ --- here $\tg(\bundle)$ is the extension of structure group of $\bundle$ by $\tg$. Note that the subvarieties $\Jac^{d,\Gamma}_{\theta,c}$ and $\Jac^{d,\Gamma}_{\theta,c'}$ are the same if $c$ and $c'$ belong to the same cohomology class in $H^2_{\theta}(\Gamma,G)$, and they are disjoint otherwise --- see the beginning of Section \ref{section-twisted-equivariant-bundles} for the definition of $H^2_{\theta}(\Gamma,G)$.

\begin{lemma}\label{lemma-Jac-fibred}
    Fix a complex scheme $S$. The groupoid $\Ff1(S)$ of families of $(\theta,c)$-twisted $\Gamma$-equivariant $G$-bundles over $X$ parametrized by $S$ is fibred over the groupoid $\Ff2(S)$ of families of $G$-bundles over $X$ parametrized by $S$ and admitting a $(\theta,c)$-twisted $\Gamma$-equivariant structure. If $S$ is contractible, for every object $L\in \Ff2(S)$, the set of isomorphism classes of elements in the fibre of $L$ is an $\HG$-torsor.
\end{lemma}
\begin{proof}
    The forgetful functor $\Ff1(S)\to \Ff2(S)$ makes $\Ff1(S)$ a fibred category over $\Ff2(S)$. Indeed, given an isomorphism of $G$-bundles $\bundle_1\to\bundle_2$ and a $(\theta,c)$-twisted $\Gamma$-equivariant structure on $\bundle_2$, the isomorphism induces a $(\theta,c)$-twisted $\Gamma$-equivariant structure on $\bundle_1$ which determines a cartesian diagramme. When $S$ is contractible, the family $\Ff2(S)$ is isomorphic to $S\times\Ff2(\pt)$ and the fibres are $\HG$-torsors by Proposition \ref{prop-G-abelian-twisted-equivariant-actions-H1}.
\end{proof}


\begin{theorem}\label{th-moduli-space-G-abelian}
    Assume that $G\cong(\C^*)^p$. Then the set of isomorphism classes of $(\theta,c)$-twisted $\Gamma$-equivariant $G$-bundles of degree $d$ over $X$ admits a fine moduli space $\tJac_{\theta,c}^{d,\Gamma}$, which is an unramified Galois cover of $\Jac^{d,\Gamma}_{\theta,c}\subset\Jac^d$ with Galois group $\HG$.
\end{theorem}
\begin{proof}
     It is well known that the abelian group $\HG$ is finite. Consider the functor $\Ff1$ sending a complex scheme $S$ to the groupoid of families of $(\theta,c)$-twisted $\Gamma$-equivariant $G$-bundles of degree $d$ over $X$ parametrized by $S$. Let $\Ff2$ be the functor sending a complex scheme $S$ to the groupoid of families of $G$-bundles of degree $d$ over $X$ parametrized by $S$, admitting a $(\theta,c)$-twisted $\Gamma$-equivariant structure. By Lemma \ref{lemma-Jac-fibred}, Assumptions (1), (2) and (3) in Section \ref{section-torsors} are satisfied, by setting $A=\HG$. The space $\Jac^{d,\Gamma}_{\theta,c}\subset\Jac^d$ is a fine moduli space representing $\Ff2$, together with a universal family which is the restriction of the Poincar\'e bundle (\ref{eq-poincare}), hence Assumption (4) holds. Assumption (5) is satisfied with $H=1$. Therefore, the theorem follows from Lemma \ref{lemma-R1-from-R2}.
\end{proof}

\subsection{Moduli space for general reductive \texorpdfstring{$G$}{G}}\label{section-moduli-G-reductive}
Let now $G$ be an arbitrary connected reductive complex Lie group. Its Lie algebra features a decomposition $\lie g=\lie g'\oplus\lie z$, where $\lie g':=[\lie g,\lie g]$ and $\lie z$ is its centre. Therefore, there exist a semisimple group $G'$ --- namely, $G'=[G,G]$ ---, an abelian group $Z$ and an unramified finite Galois covering $p_G:G\to G'\times Z$, with abelian Galois group $\Lambda:=\Gal(G/G'\times Z)$. This restricts to an unramified $\Lambda$-Galois cover $Z(G)\to Z$, where $Z(G)$ is the centre of $G$.

Note that $\theta$ and $c$ induce a homomorphism $\otheta:\Gamma\to\Aut(G'\times Z)$ and a 2-cocycle $\oc\in H^2(\Gamma,Z)$.

Since $Z$ is connected, abelian and reductive, $Z\cong(\C^*)^{\times p}$ for some $p\in\N$.
Fix an element $\cha\in i\lie z_{\lie k}^{\Gamma}$, where $\lie z_{\lie k}$ is the centre of a maximal compact subalgebra of $\lie g$. By Proposition \ref{prop-topology-parameter}, in order to parametrize $\cha$-polystable $(\theta,c)$-twisted $\Gamma$-equivariant $G$-bundles over $X$, we need to focus on $G$-bundles with topological type $\cha$. Let $d_k:=\pi_k(\cha)\in\R$. We actually assume that $d_k\in\Z$, since otherwise the moduli space is empty. Then, given a $\cha$-polystable $(\theta,c)$-twisted $\Gamma$-equivariant $G$-bundle $(\bundle,\action)$ over $X$, 
the $\C^*$-bundle $\pi_k\pi_Zp_G(\bundle)$ has degree $d_k$, where $\pi_k:Z\to \C^*$ is the projection onto the $k$-th factor. In other words, using the notation of Section \ref{section-G-abelian}, $\pi_Zp_G(\bundle)$ has degree $d=\zeta$. 

\begin{theorem}\label{th-moduli-space-G-reductive}
    Let $G$ be an arbitrary connected reductive complex Lie group and fix $\cha\in i\lie z_{\lie k}^{\Gamma}$. Then there exists a complex projective variety $\Mm(X,G,\Gamma,\theta,c)$ which is a coarse moduli space parametrizing isomorphism classes of $\cha$-polystable $(\theta,c)$-twisted $\Gamma$-equivariant $G$-bundles over $X$. It contains an open subvariety classifying isomorphism classes of $\cha$-stable $(\theta,c)$-twisted $\Gamma$-equivariant $G$-bundles over $X$.
\end{theorem}
\begin{proof}
    By \cite[Proposition 6.3]{BGGM}, the set of isomorphism classes of $(\theta,c)$-twisted $\Gamma$-equivariant $G$-bundles over $X$ is parametrized by a set $H^1_{\Gamma,\theta,c}(X,\underline{G})$, which is similar to a \v Cech cohomology set. Similarly, the set of isomorphism classes of $(\otheta,\oc)$-twisted $\Gamma$-equivariant $G'\times Z$-bundles over $X$ is parametrized by  $H^1_{\Gamma,\otheta,\oc}(X,\underline{G'\times Z})$. Extension of structure group by $p_G$ induces a map
    \begin{equation}\label{eq-extension-structure-group-pG}
        p_G:H^1_{\Gamma,\theta,c}(X,\underline{G})\to H^1_{\Gamma,\otheta,\oc}(X,\underline{G'\times Z}).
    \end{equation}
    Given an element $(\bundle,\action)\in H^1_{\Gamma,\theta,c}(X,\underline{G})$, the fibre $p_{G}^{-1}(p_{G}(\bundle,\action))$ is an $H^1_{\Gamma,\theta}(X,\Lambda)$-torsor by \cite[Proposition 6.7]{BGGM}, where $H^1_{\Gamma,\theta}(X,\Lambda)$ is finite because $\Lambda$ is finite.

    Given a $(\otheta,\oc)$-twisted $\Gamma$-equivariant $G'\times Z$-bundle $(\overline{\bundle},\action)\in H^1_{\Gamma,\otheta,\oc}(X,\underline{G'\times Z})$, the underlying $G'\times Z$-bundle $\overline{\bundle}$ determines an element $\delta(\overline{\bundle})\in H^2(X,\Lambda)$, such that $\delta(\overline{\bundle})=0$ if and only if there exists a $G$-bundle $\bundle$ such that $p_G(\bundle)\cong \overline{\bundle}$. Moreover, if $\delta(\overline{\bundle})=0$, the twisted equivariant action $\action$ on $\bundle$ lifts to a $(\theta,c')$-twisted $\Gamma$-equivariant action on $\bundle$, where $[c']\in H^2_{\theta}(\Gamma,Z(G))$ is a cohomology class such that $p_G(c')=\oc$. Multiplying by $c^{-1}$, we get a class $[c^{-1}c']\in H^2_{\theta}(\Gamma,Z(G))$. Since $c'$ and $c$ induce the same cohomology class $[\oc]\in H^2_{\otheta}(\Gamma,Z)$, the class $[c^{-1}c']$ determines an element of $H^2_{\otheta}(\Gamma,\Lambda)$. Introducing the second cohomology class $H^2_{\Gamma,\theta}(X,\Lambda)$ given by \cite[Definition 6.6]{BGGM}, we therefore find a map 
    \begin{equation*}
        \delta_c:H^1_{\Gamma,\otheta,\oc}(X,\underline{G'\times Z})\to H^2_{\Gamma,\theta}(X,\Lambda)
    \end{equation*}
    containing the data of $\delta(\overline{\bundle})\in H^2(X,A)$ and the class $[c^{-1}c']$. This satisfies that $(\overline{\bundle},\action)\in H^1_{\Gamma,\otheta,\oc}(X,\underline{G'\times Z})$ is in the image of (\ref{eq-extension-structure-group-pG}) if and only if $\delta_c(\overline{\bundle},\action)=0$. 
        
    Consider the functor $\Ff1$ sending a complex scheme $S$ to the groupoid of families of $(\theta,c)$-twisted $\Gamma$-equivariant $G$-bundles of topological type $\cha$ over $X$ parametrized by $S$. Let $\Ff2$ be the functor sending a complex scheme $S$ to the groupoid of families of $(\otheta,\oc)$-twisted $\Gamma$-equivariant $G'\times Z$-bundles of topological type $\zeta$ over $X$ parametrized by $S$, whose image by $\delta_c$ is equal to $0$.

    Let us check that Assumptions (1)-(5) in Section \ref{section-torsors} are satisfied.
    \begin{enumerate}
        \item Extension of structure group by $p_G$ determines a natural transformation $\eta:\Ff1\to\Ff2$, whose fibres are $H^1_{\Gamma,\theta}(X,\Lambda)$-torsors. Therefore, Assumption (1) holds with $A=H^1_{\Gamma,\theta}(X,\Lambda)$.

        \item Choose a complex scheme $S$, two isomorphic families of $(\otheta,\oc)$-twisted $\Gamma$-equivariant $G'\times Z$-bundles $(\overline{\bundle}_1,\action)$ and $(\overline{\bundle}_2,\action)$ on $S\times X$, a $(\theta,c)$-twisted $\Gamma$-equivariant $G$-bundle $(\bundle_2,\action)$ on $S\times X$ such that $p_G(\bundle_2,\action)=(\overline{\bundle}_2,\action)$ and an isomorphism
        \begin{equation}\label{eq-iso-bundle1-bundle2}
            (\overline{\bundle}_1,\action)\xrightarrow{\sim} (\overline{\bundle}_2,\action)
        \end{equation}
        Since $\delta_G(\overline{\bundle}_1)=0$, there exists a $G$-bundle $\bundle_1$ such that $p_G(\bundle_1)=\overline{\bundle}_1$. The total spaces of $\bundle_1$ and $\bundle_2$ are finite unramified Galois covers of the total spaces of $\overline{\bundle}_1$ and $\overline{\bundle}_2$, respectively, with the same Galois group $\Lambda$. Therefore, the isomorphism (\ref{eq-iso-bundle1-bundle2}) lifts to an isomorphism $\bundle_1\xrightarrow{\sim} \bundle_2$. If we endow $\bundle_1$ with the twisted equivariant structure induced by the twisted equivariant structure on $\bundle_2$ via this isomorphism, then we find a lift $(\bundle_1,\action)$ of $(\overline{\bundle}_1,\action)$ completing the cartesian diagramme. Therefore, $\Ff1(S)$ is a fibred category over $\Ff2(S)$, so Assumption (2) holds.

        \item The functor $\eta(S):\Ff1(S)\to\Ff2(S)$ is surjective on objects, since by construction all the objects in $\Ff2(S)$ have trivial image under $\delta_c$.

        \item Let $\ccss$ be the complex variety satisfying Corollary \ref{cor-parameter-space-semistable-twisted-equivariant}. Let $\tJac_{\otheta,\oc}^{\zeta,\Gamma}$ be the moduli space of $(\theta,c)$-twisted $\Gamma$-equivariant $Z$-bundles of degree $\zeta$ over $X$, given by Theorem \ref{th-moduli-space-G-abelian}. The map $\delta_c$ induces an algebraic morphism 
        \begin{equation*}
            \delta_c:\ccss\times \tJac_{\otheta,\oc}^{\zeta,\Gamma}\to H^2_{\Gamma,\theta}(X,\Lambda).
        \end{equation*}  
        Let $\delta_c^{-1}(0)\subset \ccss\times \tJac_{\otheta,\oc}^{\zeta,\Gamma}$. Since $H^2_{\Gamma,\theta}(X,\Lambda)$ is finite, this is a union of connected components of $\ccss\times \tJac_{\otheta,\oc}^{\zeta,\Gamma}$. It comes equipped with a universal family $(\mathcal Q\times \Pp^{\zeta},\action)$. Set $H=\GL(V)$. By Corollary \ref{cor-parameter-space-semistable-twisted-equivariant}, the complex variety $\delta_c^{-1}(0)$, together with the $\GL(V)$-action on $\delta^{-1}(0)$ induced by the action on the first factor $\ccss$ given by Corollary \ref{cor-parameter-space-semistable-twisted-equivariant}, satisfies Assumptions (4) and (5).
    \end{enumerate}

    By Lemma \ref{lemma-R1-from-R2}, there exists an unramified Galois covering $R_1\to R_2$, together with a locally universal family over $R_1$ and a linearization of the $\GL(V)$-action. Using that $\Lambda$ is finite, it is straightforward to check that a $(\theta,c)$-twisted equivariant $G$-bundle is $\zeta$-(poly, semi)-stable if and only if the corresponding $(\otheta,\oc)$-twisted equivariant $G'\times Z$-bundle is $\zeta$-(poly, semi)-stable. Since $\Mm(X,G'\times Z,\Gamma,\theta,c):=R_2\sslash\GL(V)$ is a coarse moduli space parametrizing isomorphism classes of $\cha$-polystable $(\otheta,\oc)$-twisted $\Gamma$-equivariant $G'\times Z$-bundles over $X$, the GIT quotient $\Mm(X,G,\Gamma,\theta,c):=R_1\sslash\GL(V)$ is a coarse moduli space parametrizing isomorphism classes of $\cha$-polystable $(\theta,c)$-twisted $\Gamma$-equivariant $G$-bundles over $X$. Moreover, $\Mm(X,G,\Gamma,\theta,c)$ is a Galois cover of $\Mm(X,G'\times Z,\Gamma,\theta,c)$ with Galois group $H^1_{\Gamma,\theta}(X,\Lambda)$. Since $\Mm(X,G'\times Z,\Gamma,\theta,c)$ is projective, so is $\Mm(X,G,\Gamma,\theta,c)$.
\end{proof}

\subsection{Moduli spaces of principal bundles with non-connected structure group}

According to \cite[Section 4]{BGGM}, in the particular case where $\Gamma$ acts freely on $X$, Theorem \ref{th-moduli-space-G-reductive} may be regarded as a construction of a coarse moduli space of principal bundles with non-connected structure group. Conversely, given a --- possibly non-connected --- reductive complex Lie group $\hat G$, Theorem \ref{th-moduli-space-G-reductive} provides a construction of a coarse moduli space of $\hat G$-bundles, as follows.

Let $G$ be the connected component of $\hat G$, denote by $Z$ the centre of $G$ and let $\Gamma:=\hat G/G$ be the group of connected components of $\hat G$. These fit in a short exact sequence
\begin{equation}\label{eq-general-extension}
    1\to G\to \hat G\to\Gamma\to 1.
\end{equation}
Let $a:\Gamma\to \Out(G)$ be the characteristic homomorphism of (\ref{eq-general-extension}). It is well known that there exists a lift $\Out(G)\to\Aut(G)$ of the natural surjection $\Aut(G)\to\Out(G)$ --- see \cite{BGGM}, for example ---, hence in particular there is a homomorphism $\theta:\Gamma\to\Aut(G)$ fitting in the commutative diagramme
$$
\begin{tikzcd}
\Aut(G)\arrow[r]  & \Out(G)\\
  & \Gamma\arrow[lu,dotted,"\theta"]\arrow[u,"a"]
\end{tikzcd},
$$
whose image consists of Dynkin diagramme automorphisms.
Pick such a lift $\theta$ of $a$.

\begin{definition}\label{def-twisted-product}
Given a 2-cocycle $c\in Z^2_{a}(\Gamma,Z)$, we define the \textbf{$(\theta,c)$-twisted product of $G$ by $\Gamma$}, written $G\times_{(\theta,c)}\Gamma$, to be the group which is equal to $G\times\Gamma$ as a set and has multiplication
$$(g,\gamma)(g',\gamma')=(g\theta_{\gamma}(g')c(\gamma,\gamma'),\gamma\gamma')$$
for every $g$ and $g'$ in $G$ and every $\gamma$ and $\gamma'$ in $\Gamma$.
\end{definition}

It can be seen --- see, for example, \cite[Proposition 2.2]{PNR} ---that There exists a 2-cocycle $c\in Z^2_{a}(\Gamma,Z)$ such that the extensions of $G$ given by $\hat G$ and $G\times_{(\theta,c)}\Gamma$ are equivalent, in the sense that there exists an isomorphism $\hat G\xrightarrow{\sim} G\times_{(\theta,c)}\Gamma$ fitting into
\begin{equation*}\label{eq-equivalence-extensions-G}
    \begin{tikzcd}[cong/.style = {draw=none,"\xrightarrow{\hspace*{0.2cm}\sim\hspace*{0.2cm}}" description,sloped}, eq/.style = {draw=none,"=" description,sloped}]
    1\arrow[r]  & G \arrow[r]\dar[equal] & \hat G \arrow[r]\ar[d,cong] & \Gamma \arrow[r]\dar[equal] & 1\\
    1\arrow[r]  & G \arrow[r] & G\times_{(\theta,c)}\Gamma \arrow[r] & \Gamma\arrow[r] & 1
    \end{tikzcd}.
\end{equation*}

Let $\tilde X\to X$ be a $\Gamma$-bundle, and let $\Gamma_0\le \Gamma$ be a minimal subgroup of $\Gamma$ such that $\tilde X$ admits a reduction of structure group $\tilde X_0$ to $\Gamma_0$.  By \cite[Proposition 2.10]{PNR}, the following categories are equivalent.
\begin{itemize}
    \item The category of $\hat G$-bundles $\bundle$ such that $\bundle/G\cong\tilde X$.
    \item The category of $(\theta,c)$-twisted $\Gamma_0$-equivariant $G$-bundles over $\tilde X_0$, with set of isomorphisms generated by the set of $\Gamma_0$-equivariant isomorphisms of $G$-bundles and the set of pullbacks by elements of the centralizer $Z_{\Gamma}(\Gamma_0)$ of $\Gamma_0$ in $\Gamma$.
\end{itemize}
Using this equivalence of categories and Definition \ref{def-semistable-twisted-equivariant}, we obtain the following.

\begin{definition}\label{definition-stability}
Let $\zk$ be the centre of a $\Gamma$-invariant maximal compact subalgebra $\lie k\subset\lie g$, and fix $\cha\in i\zk^{\Gamma_0}$. Let $\hat G_0:=G\times_{(\theta,c)}\Gamma_0\le\hat G$.
A $\hat G$-bundle $\bundle$ over $X$ such that $\bundle/G\cong\tilde X$, with a reduction of structure group $\bundle_0$ to $\hat G_0$, is:
\begin{itemize}
    \item \textbf{$\zeta$-stable} if $\deg \bundle_0(\sigma,s)> \pair{\zeta}s$ for any $s\in i\lie k^{\Gamma_0}$ and any reduction of structure group $\sigma\in H^0(X,\bundle_0(\hat G_0/P_s))$. Here $P_s\le\hat G_0$ denotes the parabolic subgroup defined by (\ref{eq-def-Ps}).
    \item \textbf{$\zeta$-semistable} if $\deg \bundle_0(\sigma,s)\ge \pair{\zeta}s$ for any $s\in i\lie k^{\Gamma_0}$ and any reduction of structure group $\sigma\in H^0(X,\bundle_0(\hat G_0/P_s))$.
    \item \textbf{$\zeta$-polystable} if it is semistable and, if $\deg \bundle_0(\sigma,s)=\pair{\zeta}s$ for some $s\in i\lie k^{\Gamma_0}$ and a reduction $\sigma\in H^0(X,\bundle_0(\hat G_0/P_s))$, there is a further holomorphic reduction of structure group $\sigma'\in H^0(X,(\bundle_0)_{\sigma}(P_s/L_s))$. Here $(\bundle_0)_{\sigma}\to X$ denotes the $P_s$-bundle corresponding to the reduction $\sigma$, and $L_s\le P_s$ denotes the Levi subgroup defined by (\ref{eq-def-levi}).
\end{itemize}
\end{definition}

Therefore, Theorem \ref{th-moduli-space-G-reductive} implies the following.

\begin{corollary}\label{cor-moduli-space-G-non-connected}
    Let $\hat G$ be an arbitrary --- possibly non-connected --- reductive complex Lie group, and let $\tilde X\to X$ be a $\Gamma$-bundle. Fix $\cha\in i\lie z_{\lie k}^{\Gamma_0}$, where $\zk$ is the centre of a $\Gamma$-invariant maximal compact subalgebra $\lie k\subset\lie g$. Then there is a complex projective variety $\Mm_{\tilde X}(X,\hat G)$ which is a coarse moduli space parametrizing isomorphism classes of $\cha$-polystable $\hat G$-bundles over $X$ whose quotient by $G$ is isomorphic to $\tilde X$. It contains an open subvariety classifying isomorphism classes of $\cha$-stable $\hat G$-bundles over $X$.
\end{corollary}
\begin{proof}
    Let $G$ be the connected component of the identity of $\hat G$, and let $\Gamma:=\hat G/G$ be its group of connected components. Pick $\theta$ and $c$ as above. Pick a connected component $\tilde X_0$, which is a reduction of structure group of $\tilde X$ to $\hat G_0:=G\times_{(\theta,c)}\Gamma_0\le\hat G$. Then the complex projective variety $\Mm_{\tilde X}(X,\hat G):=\Mm(X,G,\Gamma,\theta,c)/Z_{\Gamma}(\Gamma_0)$, where $Z_{\Gamma}(\Gamma_0)$ is the centralizer of $\Gamma_0$ in $\Gamma$ and $\Mm(X,G,\Gamma,\theta,c)$ is defined by Theorem \ref{th-moduli-space-G-reductive}, is the required complex projective variety --- the group $Z_{\Gamma}(\Gamma_0)$ acts on $\Mm(X,G,\Gamma,\theta,c)$ by pullback, see \cite[Proposition 2.9]{PNR}.
\end{proof}


\providecommand{\bysame}{\leavevmode\hbox to3em{\hrulefill}\thinspace}

\end{document}